\documentclass[12pt,twoside]{article}
\usepackage[T1]{fontenc}
\usepackage[utf8]{inputenc}
\usepackage{amsmath,amsthm,amssymb}
\usepackage{times}
\usepackage{enumitem}
\usepackage{amsfonts,calrsfs,color,verbatim,eucal,yfonts,mathrsfs,marginnote,mathtools,bbm}
\usepackage{hyperref}

\def\titlerunning#1{\gdef\titrun{#1}}
\makeatletter
\def\author#1{\gdef\autrun{\def\and{\unskip, }#1}\gdef\@author{#1}}
\def\address#1{{\def\and{\\\hspace*{18pt}}\renewcommand{\thefootnote}{}%
\footnote {#1}}%
\markboth{\autrun}{\titrun}}
\makeatother
\def\email#1{e-mail: #1}
\def\subjclass#1{{\renewcommand{\thefootnote}{}%
\footnote{\emph{Mathematics Subject Classification (2010):} #1}}}
\def\keywords#1{\par\medskip
\noindent\textbf{Keywords.} #1}

\newtheorem{Theorem}{Theorem}[section]
\newtheorem{Definition}[Theorem]{Definition}
\newtheorem{Proposition}[Theorem]{Proposition}
\newtheorem{Lemma}[Theorem]{Lemma}

\newtheorem{Remark}[Theorem]{Remark}
\newtheorem{Example}[Theorem]{Example}

\numberwithin{equation}{section}

\def\D{\mathbb D}
\def\R{\mathbb R}
\def\N{\mathbb N}

\def\m{\mathbf{m}}

\def\d{\text{\rm{d}}}
\def\ff{\frac}
\def\ll{\lambda}
\def\DD{\Delta}
\def\nn{\nabla}

\def\<{\langle}
\def\>{\rangle}

\newcommand{\esssup}{\operatorname{ess\,sup}}

\renewcommand{\thefootnote}{\,(\arabic{footnote})}

\newcommand{\zx}{\color{black}}
\newcommand{\yx}{\color{black}}

\newcommand{\xz}{\color{black}}
\newcommand{\xy}{\color{black}}

\begin{document}

%%%%% To ease editing, add:

\baselineskip=17pt

%%%%%%%%%%%%%%%%

%% In the running head, give an abbreviation of the title.
\titlerunning{Strong dissipativity of generalized fractional derivatives and quasi-linear (S)PDE}

\title{Strong dissipativity of generalized time-fractional derivatives and quasi-linear (stochastic) partial differential equations}

\author{Wei Liu
\and
Michael R\"{o}ckner
\and
Jos\'{e} Lu\'{i}s da Silva}
\date{}
\maketitle

\address{W. Liu: School of Mathematics and Statistics, Jiangsu Normal University, 221116 Xuzhou, China; \email{weiliu@jsnu.edu.cn} 
 \and
M. R\"{o}ckner: Faculty of Mathematics, Bielefeld University, 33615 Bielefeld, Germany / Academy of Mathematics and Systems Science, CAS, 100190 Beijing, China; \email{roeckner@math.uni-bielefeld.de} 
%(\textit{Corresponding author})
\and
 J.L. da Silva: CIMA, University of Madeira, 9020-105 Funchal, Portugal; \email{joses@staff.uma.pt}}


\begin{abstract}
In this paper strong dissipativity of generalized time-fractional derivatives on Gelfand triples of properly in time weighted $L^p$-path spaces is proved.
In particular, as special cases the classical Caputo derivative and other fractional derivatives appearing in applications  are included.
As a consequence one obtains the existence and uniqueness of solutions to evolution equations on Gelfand triples with generalized time-fractional derivatives.
These equations are of type
\begin{equation*}
 \frac{d}{dt} (k * u)(t) + A(t, u(t)) = f(t), \quad 0<t<T,
\end{equation*}
with (in general nonlinear) operators $A(t,\cdot)$ satisfying general weak monotonicity conditions.
Here $k$ is a non-increasing locally Lebesgue-integrable nonnegative function on $[0, \infty)$ with $\underset{s\rightarrow\infty}{\lim}k(s)=0$.
Analogous results for the case, where $f$ is replaced by a time-fractional additive noise, are obtained as well.
Applications include generalized time-fractional quasi-linear (stochastic) partial differential equations.
In particular, time-fractional (stochastic) porous medium  and fast diffusion  equations with ordinary or fractional Laplace operators and the time-fractional (stochastic) $p$-Laplace equation are covered.
\end{abstract}

\section{Introduction}\label{section:1}
In this paper (see Theorem \ref{thm:2.2} below) we prove existence and uniqueness of solutions to non-local in time evolution equations of type
\begin{align}\label{eq:1.1}
	\partial^{*k}_t (u-u_0) + A(t, u(t)) = f(t), \quad 0<t<T,
\end{align}
on a separable real Hilbert space ($H,\langle \cdot,\cdot \rangle_H)$, which is the pivôt space of a Gelfand triple
\begin{equation}\label{eq1.1'}
V\subseteq H (\cong H^*) \subseteq V^*,
\end{equation}
where $V$ is a reflexive Banach space with dual $V^*$.
Here $T \in (0,\infty)$, $u_0$ is the initial condition and
\begin{align*}
	A(t,\cdot) \colon V \longrightarrow V^*, \quad t \in [0,\infty),
\end{align*}
are (in general nonlinear) weakly-monotone operators satisfying \ref{cond:H1}--\ref{cond:H4} in Section \ref{section:2} below.
Furthermore, $f(t) \in V^*$, $t \in [0,\infty)$, and
\begin{align}\label{eq:1.2}
 \partial^{*k}_t u : = \partial_t(k * u):= \frac{d}{dt} \int^t_0 k(t-s)u(s) \, \mathrm{d}s, \quad t \in [0, \infty),
\end{align}
for $k\in L^1_{\mathrm{loc}} ([0, \infty))$, $k\geq 0$, non-increasing and without loss of generality right-continuous. Here we also refer to \eqref{eq:2.2} below, which is the integral form of \eqref{eq:1.1} and follows from \eqref{eq:1.1} under an additional assumption on $k$ (see condition \ref{cond:k2} in Section \ref{section:2} below).

In \cite{LRS18} \zx under more stringent conditions on A existence of solutions has been proved in \xz the special case \yx where\xy
\begin{align}\label{eq:1.3}
 k(t): = g_{1-\beta} (t):=\frac{t^{-\beta}}{\Gamma(1-\beta)}, \quad t \in [0, \infty), \beta \in (0,1),
\end{align}
i.e.,~where $\partial^{*k}_t (u-u_0)$ is the Caputo time-fractional derivative of $u$, has been treated.
For more examples of functions $k$, also called \textit{kernels} in the literature, we refer to Section \ref{section:6}.

In \cite{LRS18}, however, the stronger hypothesis that $A(t,\cdot) \colon V\longrightarrow V^*$, $t\in [0,\infty)$, is monotone (that is, $C_1=0$ in \ref{cond:H2}, see Section \ref{section:2}), was assumed, which excludes a number of important applications.
Apart from this and the more general non-local time derivatives $\partial^{*k}_t$, which for distinction we call \textit{generalized} time-fractional derivatives, in this paper we give a new and easy proof of uniqueness of solutions to \eqref{eq:1.1}.
The proofs both for generalizing to weakly-monotone $A(t,\cdot)$, $t\in [0,\infty)$, and for uniqueness turn out to be consequences of a new result on (generalized) time-fractional derivatives \yx in this paper\xy. \yx This is that, we identify \xy$\yx-\xy\partial^{*k}_t$ as a generator of a $C_0$-operator semigroup  on a properly in time weighted $L^2$-space \yx and prove that it \xy is strongly dissipative \yx(see Proposition \ref{eq:3.1} and Lemma \ref{eq:3.4'} below, as well as their consequence Theorem \ref{thm:2.1}). This \xy together with its applications to uniquely solving \eqref{eq:1.1} (see Theorem \ref{thm:2.2}  and Section \ref{section:4} below) can be considered as the main contribution of this work.
In particular, our results are applicable to the time-fractional generalized porous medium  and fast diffusion equations with ordinary or fractional Laplace operators
\begin{equation*}
\partial_t^{*k} ( u(t) -u_0) +(-\Delta)^\alpha (| u(t) |^{r-1} u(t) )= f(t, u(t) )\end{equation*}
and the time-fractional $p$-Laplace equation
\begin{equation*}
 \partial_t^{*k} ( u(t) -u_0) -  {\rm div}\left( |\nabla  u(t) |^{p-2} \nabla  u(t)  \right) = f(t, u(t)).
\end{equation*}
We refer to Section \ref{section:7} for details and more general types of these equations, which our results apply to \zx and which are not covered by results in the literature.\xz

\yx As \xy a consequence by a simple shift argument we obtain the unique solvability of the stochastically perturbed variant of \eqref{eq:1.1}, namely
\begin{align}\label{eq:1.4}
 \partial^{*k_1}_t (X(t)-X_0) + A(t, X(t))= \partial^{*k_2}_t \int^t_0 B(s) \, \mathrm{d}W(s), \quad 0<t<T,
\end{align}
where $W(t)$, $t\geq 0$, is a cylindrical Brownian motion in some other separable Hilbert space $(U, {\langle\cdot, \cdot \rangle}_U)$ and $B(s) \colon U \longrightarrow H$ is a Hilbert--Schmidt operator for every $s\in [0, \infty)$ (see Theorem \ref{thm:2.3} below).

At this point we would like to stress that, since the operator $A$ is allowed to be nonlinear as e.g.~a quasi-linear partial or pseudo differential operator (see Section \ref{section:7} below for examples), the classical probabilistic ``inverse subordination method'' (see \cite{BM01,BMN09,MNV09,OB09} and also \cite{Che17,CKKW} as well as the references therein) to solve equation \eqref{eq:1.1} does not work.
\zx

Let us now explain our method of proof in more detail and in comparison with the usual method in papers on time-fractional differential equations by other authors. The first main point is that we do not solve (as is commonly done in the literature) the integral equation corresponding to \eqref{eq:1.1}, that is, \eqref{eq:2.2} below. This, by the way, would require an additional condition on $k$ (see Theorem \ref{thm:2.2}(ii)).
Instead, we solve equation \eqref{eq:1.1} directly. The reason is that for \eqref{eq:2.2} we cannot exploit the weak monotonicity and coercivity assumptions, (H2), (H3) respectively, on $A$, because of the convolution integral on the right hand side of \eqref{eq:2.2}. Therefore, the idea to find a solution to \eqref{eq:1.1} is to show that the map on its left hand side, considered as a map from paths to paths, is surjective from $\mathcal{V}$ to $\mathcal{V}^*$ in a suitable Gelfand triple $\mathcal{V}\subset \mathcal{H}\subset \mathcal{V}^*$ of $L^p$-path spaces, (see \eqref{eq:2.33} below).
It follows by the assumptions (H1)-(H4) and assuming $C_1=0$ in (H2), that, if $\mathcal{A}$ denotes the map on paths given by $A$ (see \eqref{eq:2.4'}), then $\mathcal{A}$ alone has this surjectivity property, because under these conditions $\mathcal{A}\colon \mathcal{V}\rightarrow \mathcal{V}^*$ is maximal monotone and coercive. But it is a highly non-trivial question, whether then also the sum $\mathcal{A}+\partial_t^{*k}$ is surjective onto $\mathcal{V}^*$. To prove the latter we prove that $\partial_t^{*k}$ is the infinitesimal generator of a (linear) $C_0$-semigroup $(U_t^k)_{t\geq 0}$ on the pivot space $\mathcal{H}=L^2([0,\infty);H)$ of the Gelfand triple \eqref{eq:2.33} (see Proposition \ref{proposition:3.1} and \ref{corollary:3.2}), which is given explicitly by \eqref{eq:2.4}.
Here it is crucial to take the whole time interval $[0,\infty)$ in the definition of $\mathcal{H}$ rather than just $[0,T]$, in contrast to what one would expect, because one wants to solve \eqref{eq:1.1} only for $0<t<T$. Since the restriction of $(U_t^k)_{t\geq0}$ to $\mathcal{V}$ is again a $C_0$-semigroup, by a non-standard (see Remark \ref{remark:a.2}(iii)) perturbation result (see Theorem \ref{eq:4.1}), we can conclude that on some specific domain $\mathcal{F}^k$ (=generalized time-fractional Sobolev space) we have $\mathcal{A}+\partial_t^{*k}\colon \mathcal{F}^k\subset \mathcal{V}\rightarrow \mathcal{V}^*$ is surjective.
For this, however, we need that $A$ is monotone (i.e. $C_1$ in (H2) must be zero). To reduce our case (i.e. $C_1\geq0$) to this case the strict dissipativity of $\,-\partial_t^{*k}$ on the time weighted Gelfand triple $\mathcal{V}_\gamma\subset \mathcal{H}_\gamma\subset \mathcal{V}_\gamma^*$ (see \eqref{eq:2.8} below), where $dt$ is replaced by $e^{-\gamma t}dt$, $\gamma>0$, proved in this paper (with explicit dissipativity constant $\psi_k(\gamma)$, where $\psi_k$ is the Bernstein function with Levy measure $M^k$, whose distribution function is $k$; see Lemma \ref{eq:3.4'} and \eqref{eq:2.9}), becomes crucial.
As another consequence of the strict dissipativity of $\,-\partial_t^{*k}$ we get uniqueness of solutions to \eqref{eq:1.1} in a very easy and standard way (see the end of the proof of Theorem \ref{thm:2.2}(i) in Section 4). To the best of our knowledge this proof is completely new in the case of generalized time-fractional derivatives, as is the result that the latter are all strictly dissipative on appropriately time weighted Gelfand triples as above.
In our paper \cite{LRS18} on the special case, where in \eqref{eq:1.1} $\partial_t^{*k}$ is the classical Caputo derivative $\partial_t^\beta$, $\beta\in(0,1)$, we also proved existence of solutions to \eqref{eq:1.1} (and not to its corresponding integral version \eqref{eq:2.2}) by showing the surjectivity of the map on its right hand side, but as mentioned above, under more stringent conditions on $A$. There, however, we could not prove uniqueness by this approach because of the lack of strict dissipativity of the Caputo derivative, which we only have now as a special case of one of the main results in this paper.
\xz

\zx Next we would like to make some historical remarks, explain the motivation to study equations as \eqref{eq:1.1} and comment on the relation of our results with those in the literature.

Fractional calculus has a long history. Its origins can be traced back to the end of the seventeenth century (cf.~\cite{R77}), \zx and it has been experiencing an impressive revival in
\xz the last few decades. One of the main reasons is that scientists  and engineers
have established a vast amount of new models (e.g. to describe anomalous diffusions) that naturally involve time-fractional differential equations, which have been applied successfully, e.g.~in mechanics (cf.~\cite{M10}), bio-chemistry (cf.~\cite{DE88, F80}), electrical
engineering (cf.~\cite{Di10}), medical science (cf.~\cite{CRA08}). For more applications and references we refer to \cite{BDS,He11,MS12,MBK99,MK00,T20}.\xz

There is \yx a lot of \xy motivation from \zx both \xz Physics and Mathematics as regards
the use of \textit{\yx generalized\xy} time-fractional derivatives (see e.g.\cite{ACV16,KKL17,MBK99,MK00,MS12}). Here we mention
a few examples. Starting from the seminal paper \cite{Caputo1967}
the Caputo fractional derivative was introduced to properly handle
initial value problems, namely to model waves in viscoelastic media.
Later on it was generalized to the so called distributed order derivative
(also called variable order derivative in \cite{Lorenzo2002}), see
\cite{Caputo1995} and Example \ref{exa:distr-order-deriv} below
for details. Other successful applications of the distributed order
derivative \yx include \xy the kinetic theory (cf.~\cite{CGSG2003,Chechkin2003,Kochubei2008,Kochubei09})
to describe ultra-slow diffusion or the theory of elasticity (see
\cite{Lorenzo2002}) for the description of rheological properties
of composite materials. Inverse stable subordinators arise (cf.~\cite{Meerschaert2002,MS2004})
as scaling limits of continuous time random walks. In \cite{Meerschaert2006}
it was shown that under certain technical conditions the probability
density of the hitting time process $E(t)$ (that is the inverse of
a certain subordinator) solve a distributed order time-fractional
evolution equation. For more applications of the distributed order derivative
we refer the reader to \cite{Atanackovic2009,Caputo2001,Gorenflo2006,Kochubei2008a,Mainardi:2007uq,Mainardi:2008fk}.

When dealing with a particular anomalous diffusion process, it is
often difficult to choose which model of the time-fractional diffusion
equations is suitable for its mathematical description. Thus a general
framework of time-fractional derivatives is needed. In \cite{Kochubei11},
the author introduced a general fractional calculus for integral
 operators of convolution type with an arbitrary nonnegative locally
integrable kernel $k$. He considered the initial value problem for
both relaxation and diffusion equations with these general time-fractional
derivatives. Since then many authors applied the generalized time-fractional
derivative to solve in general linear fractional equations and nonlinear
differential equations, see e.g. \cite{Luchko2016,Zacher2015,LRS18}
and references therein. We want to remark that  a  huge amount of  the  existing literature  on  this subject  concentrates on  the case of linear and semilinear type equations. However, to the best of our knowledge, there are only very few  results  that  are  applicable to the quasilinear case, to which the results in this paper have their main new applications.

We should mention that  time-fractional linear evolution equations in the Gelfand triple setting have first been investigated in \cite{Za09}. Later on the author
also proved the global solvability of a nondegenerate parabolic equation with time-fractional derivative in \cite{Za12} (cf. \cite{ACV16} for more general cases).
However, these results cannot be applied to
quasilinear type equations like the porous medium or the $p$-Laplace equation. In \cite{JP04} the authors investigate
elliptic-parabolic integro-differential equations with
$L^1$-data. Their framework includes the time-fractional $p$-Laplace equation. However, the authors in \cite{JP04} only obtain generalized solutions
($i.e.$ entropy solutions). Therefore, the results of the current paper generalize or complement  the corresponding results in \cite{JP04,LRS18,Zacher2015,Za09,Za12}
within the \yx general \xy setting of time-fractional quasilinear PDE\zx s \xz with weakly monotone coefficients. In particular,
the authors in \cite{Zacher2015} derive very interesting decay estimates for the solutions of
time-fractional porous medium and $p$-Laplace equations (by assuming the existence of solutions), and
the decay behaviour is notably different from the case with usual time derivative. In \cite{LRS18}, we
give a positive answer to the  question on the existence and
uniqueness of solutions to the time-fractional porous medium equations and $p$-Laplace equations, which are left open in \cite{Zacher2015}. The current work further extend the results in \cite{LRS18} to  both generalized fractional derivative and the weakly monotone case.

Recently, there has been also growing interest in  time-fractional \textit{\yx stochastic \xy} partial differential equations.
For instance, the authors in \cite{CKK,KKL19} investigate the $L^2$-theory and Sobolev space theory respectively for a class of semilinear SPDEs with time-fractional
 derivatives, which can be used to describe
random effects on transport of particles in media with thermal memory, or particles subject to sticking and
trapping.
In  \cite{MN17,MN15}, the authors consider a space-time fractional stochastic heat type equation to model phenomena with random
effects with thermal memory, and they prove the existence and uniqueness of mild solutions as well as some intermittency property.
For a linear stochastic partial differential equation of fractional order both in the time and space variables with a different type of noise term,  we refer to \cite{CHHH} (see also \cite{AX17,DL13}).
In \cite{C17} the authors investigate linear stochastic time-fractional partial differential equations for the type of heat equation and wave equation.

The list of references quoted above is far from being complete, but show the enormous interest in the subject. \yx However none of them contains results on quasi-linear SPDEs with fractional or generalized fractional time derivative, whereas these form a class of equations to which the results of the present paper apply. \xy

The rest of the paper is organized as follows. In Section 2 we present the main results (Theorems \ref{thm:2.1}, \ref{thm:2.2} and \ref{thm:2.3}) on the existence and uniqueness of solutions to
deterministic and stochastic nonlinear evolution equations with generalized time-fractional derivatives.
Theorem \ref{thm:2.1} will be proved in Section \ref{section:3}.
The proof of Theorem \ref{thm:2.2} is given in Section \ref{section:4}.
It relies on Theorem \ref{thm:2.1} and an abstract perturbation result (see Theorem \ref{eq:4.1}).
\yx Since this is not standard, for \xy the convenience of the reader we include its proof in the Appendix of this paper.
Because of its importance we give a more detailed proof than the \zx very sketchy one \xz in \cite{LRS18}.
The proof of Theorem \ref{thm:2.3} will be given in Section \ref{section:5}$\,$. Section \ref{section:6} \yx contains \xy examples of kernels $k$ which appeared in literature. In Section \ref{section:7} we apply the main results to some concrete quasi-linear deterministic and stochastic PDE\zx s\xz.

\section{Framework and main results}\label{section:2}
Let  $(H, \<\cdot,\cdot\>_H)$ be a real separable Hilbert space identified with its dual space $H^*$ by the Riesz isomorphism.
Let $V$ be  a real reflexive  Banach space, continuously and densely embedded into $H$.
Then we have the following Gelfand triple
$$V \subseteq H \cong  H^*\subseteq V^*.$$
Let ${ }_{V^*}\<\cdot,\cdot\>_V$ denote the dualization between  $V$ and its dual space $V^*$ and let $\| \cdot \|_{H}$, $\| \cdot \|_{V}$, $\| \cdot \|_{V^*}$ denote the respective norms.
Then it is easy to show that
$$ { }_{V^*}\<u, v\>_V=\<u, v\>_H, \ \  u\in H ,v\in V.$$
Now, for $T \in [0, \infty)$ fixed, we consider the following general nonlinear evolution equation with generalized time-fractional derivative
\begin{align}\label{eq:2.1}
	\partial^{*k}_t (u-u_0) + A(t, u(t)) = f(t),\ \text{for $\mathrm{d}t$-a.e. $t\in[0,T ]$},
\end{align}
where $k \in L^1_{\mathrm{loc}} ([0,\infty);\mathbb{R}, \mathrm{d}s) =: L^1_{\mathrm{loc}} ([0,\infty))$ (with $\mathrm{d}s$ = Lebesgue measure) satisfies condition \ref{cond:k} below, $f\in L^1 ([0,\infty); V^*)$, $\partial^{*k}_t$ is as in \eqref{eq:1.2}, $u_0 \in V$ is the initial condition, and we are seeking for solutions  $u\in L^1([0,\infty);V)$.
Therefore, the derivative $\frac{\mathrm{d}}{\mathrm{d}t}$ in the definition \eqref{eq:1.2} of $\partial^{*k}_t$ is understood in the weak sense.
Consider the following conditions on $k$:
\begin{enumerate}[label=(k), leftmargin=1.5cm]
\item \label{cond:k}  $k \in L^1_{\mathrm{loc}} ([0,\infty))$, $k$ is nonnegative, non-increasing and (hence without loss of generality) right continuous such that $\underset{s\rightarrow\infty}{\lim}k(s)=0$.
\end{enumerate}
\begin{enumerate}[label=($\tilde{\textrm{k}}$),leftmargin=1.5cm]
\item \label{cond:k2} There exists $\tilde{k} \in L^1_{\mathrm{loc}} ([0,\infty))$, nonnegative, such that
	$$(\tilde {k} * k) (t)= \int^t_0 \tilde {k} (t-s) k(s) \, \mathrm{d}s = 1 \qquad \text{for $\mathrm{d}t$-a.e. } t\in  [0,\infty).$$
\end{enumerate}
Here and below we consider $k$ and $\tilde{k}$ as functions on $\R$ defining them to be zero on $(-\infty,0)$.
Obviously \ref{cond:k} and \ref{cond:k2} hold for $k$ as in \eqref{eq:1.3}.

If \ref{cond:k} and \ref{cond:k2} hold, then \eqref{eq:2.1} can be rewritten as
\begin{align}\label{eq:2.2}
  u(t) = u_0 -\int^t_0 \tilde{k}(t-s) A(s,u(s)) \, \mathrm{d}s + \int^t_0 \tilde{k} (t-s) f(s)\, \mathrm{d}s \qquad \text{for } \mathrm{d}t\text{-a.e.}~t \in [0,\infty).
\end{align}
This can be easily seen by first integrating \eqref{eq:2.1} with respect to $\mathrm{d}t$ and using the fact that the convolution with $k * \tilde{k} = \tilde{k} * k$ is just integration with respect to $\mathrm{d}t$.
Defining $\tilde{u} (t)$ to be equal to the right hand side of \eqref{eq:2.2} for \textit{every} $t\in [0,\infty)$, we have that $\tilde{u}$ is a $\mathrm{d}t$-version of $u$, hence still satisfies \eqref{eq:2.2} with $\tilde{u}(0) = u_0$.
In this sense $u$ has $u_0$ as its initial condition.
Now let us specify the conditions on the map
\begin{align*}
	A \colon [0,\infty) \times V \longrightarrow V^*
\end{align*}
 which is first of all assumed to be $\mathcal{B}([0,\infty)\times V)/\mathcal{B}(V^*)$ measurable (where $\mathcal{B} (\cdot)$ means Borel $\sigma$-algebra of $\cdot$) and assumed to satisfy the following:
 There exist $\alpha \in (1,\infty)$, $\delta \in (0,\infty)$, $C_1, C_2 \in [0,\infty)$ and $g\in L^1 ([0,\infty); \mathbb{R})$ such that for all $t\in [0,\infty)$, $v, v_1, v_2\in V$
\begin{enumerate}[label=(H\arabic*), leftmargin=1.5cm]
 \item \label{cond:H1}(Hemicontinuity)
      The map  $ s\mapsto { }_{V^*}\<A(t,v_1+s v_2),v\>_V$ is  continuous on $\mathbb{R}$.

\item \label{cond:H2} (Weak Monotonicity)
     $$  { }_{V^*}\<A(t,v_1)-A(t, v_2), v_1-v_2\>_V \ge -C_1 \| v_1 -v_2\|_H^2. $$

\item \label{cond:H3} (Coercivity)
    $$  { }_{V^*}\<A(t,v), v\>_V  \ge  \delta \|v\|_V^{\alpha} -C_2\|v\|_H^2 - g(t).$$

\item \label{cond:H4} (Growth)
   $$ \|A(t,v)\|_{V^*}^{\frac{\alpha}{\alpha-1}} \le  g(t) + C_2 \left(\|v\|_V^{\alpha}+\|v\|_H^2 \right).$$
\end{enumerate}
We define the following  spaces,
\begin{align}\label{eq:2.33}
  \mathcal{V}&=L^\alpha([0,\infty); V)\cap L^2([0,\infty); H),\notag\\
  \mathcal{H}&=L^2([0,\infty); H),\\
  \mathcal{V}^*&=L^{\frac{\alpha}{\alpha-1}}([0,\infty); V^*)+L^2([0,\infty); H)\notag,
\end{align}
where $\| \cdot \|_{\mathcal V} := \max ( \| \cdot \|_{L^\alpha([0,\infty); V)} , \| \cdot \|_{\mathcal H} )$ and for $u \in \mathcal V^*$
\begin{align*}
\| u \|_{\mathcal{V}^*} := \inf \left\{ \| u_1 \|_{L^{\frac{\alpha}{\alpha-1}}([0,\infty); V^*)} + \| u_2 \|_{\mathcal H}:
u_1 \in L^{\frac{\alpha}{\alpha-1}}([0, \infty); V^*), u_2 \in \mathcal H \ s.t. \ u=u_1+u_2\right\}.
\end{align*}
Then for $u_0=0$ (the case for general initial conditions $u_0 \in V$ will then follow easily as we shall see below)
the original equation \eqref{eq:2.1} can be rewritten in the following form
\begin{equation}\label{eq:2.2'}
 \partial_t^{*k}u+\mathcal{A}u=f,
\end{equation}
where
\begin{equation}\label{eq:2.4'}
	\mathcal{A} \colon \mathcal{V}\longrightarrow \mathcal{V}^*; (\mathcal{A}u)(t)=A(t,u(t)), \ t\in[0,\infty).
\end{equation}
It is easy to see that $\mathcal{A} \colon \mathcal{V}\longrightarrow \mathcal{V}^*$ is weakly monotone, coercive and bounded on bounded sets.
Below we fix $k$ and $A$ as above.

To formulate our main results we furthermore need to define the following ``shift to the right'' semigroup $U_t$, $t > 0$, on $\mathcal{H}$.
Below we extend every $f \in \mathcal{H}$ by $f:=0$ on $(-\infty, 0)$ to a function $f \colon \mathbb{R} \longrightarrow H$.
For $f \in \mathcal{H}$, $t \geq 0$, define
\begin{align}\label{eq:2.3}
	U_t f(r) := \mathbbm{1}_{[0, \infty)} (r - t) f (r-t), \quad r \in [0, \infty).
\end{align}
Then it is trivial to check that $(U_t)_{t > 0}$ is a strongly continuous (shortly: $C_0$-)contraction semigroup on $\mathcal{H}$ and it obviously can be restricted to a $C_0$-semigroup on $\mathcal{V}$ (even in this case consisting also of contractions on $\mathcal{V}$).
Now for $k$ as above and $\mu_t^k$, $t \geq 0$, as defined in \eqref{eq:3.4} below, we define for $f \in \mathcal{H}$
\begin{align}\label{eq:2.4}
 U_t^k f := \int\limits_{0}^{\infty} U_s f\  \mu_t^k( \mathrm{d}s)= f * \mu_t^k, \quad t \geq 0,
\end{align}
It is a well-known fact (see e.g. \cite[Chap.~II, Sect.~4b]{MR92}), that $(U_t^{k})_{t > 0}$ is also a $C_0$-semigroup of contractions on $\mathcal{H}$.
Let $\Lambda^k$ with domain $D(\Lambda^k, \mathcal H)$ be its infinitesimal generator on $\mathcal H$.

Obviously, $(U^k_t)_{t>0}$ can be restricted to a $C_0$-semigroup on $\mathcal{V}$ (again consisting of contractions).
The generator of the latter is again $\Lambda^k$, but with domain
\begin{align*}
  D(\Lambda^k, \mathcal{V}): = \{u\in D (\Lambda^k, \mathcal{H})\cap\mathcal{V} \mid \Lambda ^k u \in \mathcal{V}\}.
\end{align*}
Then $D (\Lambda^k, \mathcal{V})$ is dense in $\mathcal{V}$, hence so is $D(\Lambda^k,\mathcal{H}) \cap \mathcal{V}$.\\

By \cite[Lemma 2.3]{St99}, $\Lambda^k \colon D(\Lambda^k, \mathcal H) \cap \mathcal V \longrightarrow \mathcal V^*$ is closable as an operator from $\mathcal V$ to $\mathcal V^*$.
\yx We denote its closure again by $\Lambda^k$ and the domain of the latter by $\mathcal F^k$. Then \xy $\mathcal F^k$ is a Banach space with norm $\| u \|_{\mathcal F^k} := (\|u\|^2_{\mathcal V} + \| \Lambda^k u\|^2_{\mathcal V^*})^{\frac12}$, $u \in \mathcal F^k$.
We would like to mention here that, as will be seen in the applications in Section \ref{section:7}, $\mathcal F^k$ is a generalization of a space-time Sobolev space with generalized time-fractional derivative.
It will turn out (see Theorem \ref{thm:2.2} below) that it is the appropriate space in which equation \eqref{eq:2.1} can be solved.

Finally, we define a convenient domain of $\partial^{*k}_t$, namely:
\begin{align}\label{eq:2.5}
D(\partial^{*k}_t):= \{u\in  \mathcal V^*  \mid k*u\in W^{1,1} ((0,T);V^*),  ~ \forall \, T \in (0,\infty) \},
\end{align}
where $W^{1,1} ((0,T); V^*)$ denotes the standard Sobolev space of order $1$ in $L^1([0,T];V^*)$.

We recall that for $u \in L^1([0,\infty); V^*)$ we set $u\equiv0$ on $(\infty,0)$, hence
\begin{align}\label{eq:2.6}
	(k * u) (t) = \int_0^t k (t-s)  u(s) \, \mathrm{d}s.
\end{align}
Then, obviously, for all $T \in (0, \infty)$
\begin{align}\label{eq:2.7'}
 k * u = (\mathbbm{1}_{[0,T]} k) * (\mathbbm{1}_{[0,\infty)} u) \quad \text{ on } [0,T],
\end{align}
where  for $p \in [1, \infty)$ the function on the right hand side belongs to $L^p(\R;V^*)$ if so does $\mathbbm{1}_{[0,\infty)} u$,
and is in $C(\R; V^*)$, if in addition $k \in L^{p'}_{\mathrm{loc}}([0, \infty))$, where $p':=\frac{p}{p-1}$.

For $\gamma \in (0,\infty)$, $p\in[1,\infty)$, and $E=\mathbb{R},V,V^*$ or $H$ we set
\begin{align}\label{eq:2.7''}
	L^p_\gamma ( [0, \infty); E) := L^p( [0, \infty) ; E, e^{-\gamma t} \mathrm{d}t),
\end{align}
\zx and define $\mathcal{V}_\gamma$, $\mathcal{H}_\gamma$, $\mathcal{V}_\gamma^*$ as in \eqref{eq:2.33} with Lebesgue measure $dt$ replaced by $e^{-\gamma t}dt$.\xz

\begin{Theorem}\label{thm:2.1}
Suppose that $k$ satisfies \ref{cond:k}. Then:
\begin{enumerate}[label=(\roman*), leftmargin=1.5cm]
\item $\mathcal F^k \subset D(\partial^{*k}_t)$ and
		$$\Lambda^k u =- \partial^{*k}_t u \yx\quad for\ all\ u\in\mathcal{F}^k.\xy$$
		In particular, $k*u\in C([0,\infty),H)$.
\item (``strong dissipativity \zx in $\mathcal{V}_\gamma\subset \mathcal{H}_\gamma\subset \mathcal{V}^*_\gamma$ \xz''). For every $\gamma \in (0,\infty)$ and all $u \in \mathcal F^k$
		\begin{align}\label{eq:2.8}
			\int_0^\infty {}^{}_{V^*}\langle \partial^{*k}_t u(s), u(s) \rangle_{V}^{} \, e^{-\gamma s} \, \mathrm{d}s \geq \frac12 \psi^k(\gamma) \int_0^\infty \|u(s)\|^2_H e^{-\gamma s} \, \mathrm{d}s,
		\end{align}
		where
		\begin{align}\label{eq:2.9}
			\psi^k(\gamma) := \int_{(0, \infty)} (1 - e^{-\tau \gamma}) M^k (\mathrm{d} \tau) \ (>0~!)
		\end{align}
		and $M^k$ is the unique measure on $( (0, \infty), \mathcal B(0, \infty) )$ such that $k(s) = M^k( (s,\infty) )$, $s\in(0,\infty)$ (see the beginning of Section \ref{section:3} for more details, in particular \eqref{eq:3.2}).
\end{enumerate}
\end{Theorem}

The proof of Theorem \ref{thm:2.1} will be given in Section \ref{section:3} below.
We only mention here that assertion (i) is easy to prove for sufficiently smooth functions.
The point here is that it holds for all $u \in \mathcal F^k$.
The proofs of the following two theorems are contained in Section \ref{section:4} below.

\begin{Theorem}\label{thm:2.2}
Suppose that $T \in [0,\infty)$, $k$ satisfies \ref{cond:k} and $A \colon [0, \infty) \times V \longrightarrow V^*$ satisfies \ref{cond:H1}--\ref{cond:H4}.
Furthermore, assume that for $C_1$ from \ref{cond:H2} there exists $\gamma \in (0, \infty)$ such that $\psi^k(\gamma) > 2 C_1$, which is always the case if $\lim\limits_{s \rightarrow 0} k(s) = \infty$.
 Then:
 \begin{enumerate}[label=(\roman*)]
 	\item For every $u_0 \in V$ and $f \in \mathcal{V}^*$, \eqref{eq:2.1} has a unique solution $u$ such that $u - u_0 \varphi \in \mathcal{F}^k$ for every $\varphi \in L^\alpha([0,\infty);  \mathbb{R})$ with $\varphi \equiv 1$ on $[0, T+1)$. In particular,
\begin{equation}\label{eq:2.10}
  u-u_0 \varphi\in L^\alpha([0,\infty); V); \   \partial_t^{*k} (u- u_0 \varphi) \in L^{\frac{\alpha}{\alpha-1}}([0,\infty); V^*)
\end{equation}
and $t \mapsto \int_0^t k (t-s) (u(s) - u_0 \varphi(s)) \, \mathrm{d}s$ has a continuous $H$-valued $\mathrm{d}t$-version.
\item If, in addition, \ref{cond:k2} holds, then for $\mathrm{d}t$-a.e.~$t \in [0, T]$,
\begin{equation}\label{eq:2.11}
 u(t)= u_0 - \int_0^t \tilde{k} (t-s) A(s,u(s))\, \mathrm{d}s +  \int_0^t \tilde{k} (t-s) f(s) \, \mathrm{d}s .
\end{equation}
Furthermore, if $\tilde k \in L^\alpha_{\mathrm{loc}}([0, \infty))$,  $t \mapsto u(t)$ has a  continuous $V^*$-valued $\mathrm{d}t$-version.
\end{enumerate}
\end{Theorem}
Now we turn to our last main result, namely the stochastic version of \eqref{eq:2.1} and \eqref{eq:2.2}.\\
Suppose that $U$ is a Hilbert space and $W(t)$  is  a $U$-valued cylindrical Wiener process  defined on a filtered probability space $(\Omega,\mathcal{F},(\mathcal{F}_t)_{t\geq0},\mathbb{P})$ with normal filtration $\mathcal{F}_t$, $t\geq0$.
Now we consider stochastic nonlinear evolution equations with generalized time-fractional derivative of  type
\begin{equation}\label{eq:2.12}
 \partial_t^{*k_1}(X(t)-x_0)+A(t,X(t))=\partial_t^{*k_2} \int_0^t B(s) \, \mathrm{d} W(s),  \quad 0<t<T,
\end{equation}
where $x_0\in V$ and $B \colon [0,T]\longrightarrow L_{HS}(U; H)$ is measurable, here $(L_{HS}(U; H), \|\cdot\|_{HS})$ denotes the space of all Hilbert--Schmidt operators from $U$ to $H$.
Note that, if $k_1$ satisfies \ref{cond:k2}, the integral form of \eqref{eq:2.12} is as follows
\begin{align}
 X(t)= x_0 - \int_0^t \widetilde{k_1} (t-s) A(s,X(s))\, \mathrm{d}s +  \int_0^t  (\widetilde{k_1}*k_2) (t-s)  B(s) \, \mathrm{d} W(s).
\end{align}
For this we need to assume more about $k_1$ and $k_2$ from above, namely that they satisfy
\begin{enumerate}[label=(ks), leftmargin=1.5cm]
\item \label{cond:ks} ~ \ref{cond:k} holds for $k_1$ and $k_2$, and $k_1$ satisfies \ref{cond:k2} such that $\widetilde{k_1}*k_2  \in L^2_{\mathrm{loc}} ([0,\infty))$.
\end{enumerate}
Note that the stochastic integral term in \eqref{eq:2.12}
  $$   F(t):=\int_0^t (\widetilde{k_1}*k_2) (t-s)  B(s) \, \mathrm{d} W(s)  $$
 is well-defined if e.g.~$\|B\|_{HS} \in L^\infty_{\mathrm{loc}} ([0,\infty))$, because then
$$  \int_0^t (\widetilde{k_1}*k_2)^2 (t-s) \|B(s)\|_{HS}^2 \, \mathrm{d}s < \infty. $$
If $k_1=k_2$, then the stochastic integral term is even well-defined if merely
$\|B\|_{HS} \in L^2_{\mathrm{loc}} ([0,\infty))$.
\begin{Theorem}\label{thm:2.3}
Suppose that \ref{cond:ks} holds, $A$ satisfies \ref{cond:H1}--\ref{cond:H4} and $B\in L^\infty([0,T], L_{HS}(U; H))$.
Assume also that $F\in V, \mathrm{d}t\otimes \mathbb{P}$-a.e.~(which is e.g.~the case if $B(t)$ is a Radonifying map from $U$ to $V$).
Then:
\begin{enumerate}[label=(\roman*)]
\item
For every $x_0\in V$ the ``shifted equation''
\begin{align}\label{eq:2.13}
	\partial_t^{*k_1} (X(t)-F(t)-x_0 \varphi) + A(t, X(t)) =0 , \quad \text{$\mathrm{d}t$-a.e. $t\in[0,T]$,}
\end{align}
has a unique $(\mathcal{F}_t)$-adapted solution $X$ such that $X-F-x_0\varphi \in \mathcal F^{k_1}$, $\mathbb P$-a.s.~for every $\varphi \in L^\alpha([0,\infty) ; \R)$ with $\varphi \equiv 1$ on $[0,T+1)$.
In particular,
$$  X-F-x_0\varphi \in L^\alpha([0,\infty); V); \   \partial_t^{*k_1} (X-F-x_0\varphi)\in L^{\frac{\alpha}{\alpha-1}}([0,\infty); V^*), \  \mathbb{P}\text{-a.s.}  $$
and $t \mapsto \int_0^t k_1 (t-s) (X(s) - x \varphi(s)) \, \mathrm{d}s\ \mathbb{P}$-a.s. has a continuous $H$-valued $\mathrm{d}t$-version.
\item For $\mathrm{d}t$-a.e.~$t \in [0, T]$,
\begin{align}\label{eq:2.14}
 X(t)= x_0 - \int_0^t \widetilde{k_1} (t-s) A(s,X(s))\, \mathrm{d}s +  \int_0^t (\widetilde{k_1}*k_2) (t-s) B(s) \, \mathrm{d} W(s),\ \mathbb{P}\text{-a.s.}
\end{align}
Furthermore, if $\widetilde{k_1} \in L^\alpha_{\mathrm{loc}} ([0, \infty))$, $t \mapsto X(t)\ \mathbb{P}$-a.s. has a  continuous $V^*$-valued $dt$-version.
\end{enumerate}
\end{Theorem}
\begin{Remark} In \cite{LRS18}, we have investigated the case that $ \partial_t^{*k_1}=\partial_t^{\beta},   \partial_t^{*k_2}=\partial_t^{\gamma} $ and $A$ is monotone. Then it is easy to see that
the assumption  $(ks)$ is equivalent to $\gamma<\beta+\frac{1}{2}$. We want to remark that
the special case $\gamma=\beta$  or $\gamma=1$ has been intensively investigated for some semilinear SPDE models (such as the stochastic heat equation or the stochastic wave equation), see e.g. \cite{AX17,C17,CHHH,MN17,MN15} and more references therein. 
It's easy to see that we can also have fractional Brownian motion or L\'{e}vy process as the noise in (\ref{eq:2.12}).
\end{Remark}

\section{Generalized time-fractional derivatives as generators of $C_0$-semigroups and their strong dissipativity}\label{section:3}
In this section we prove Theorem \ref{thm:2.1}, so assume that $k$ satisfies $(k)$. By Caratheodory's theorem there exists a $\sigma$-finite (nonnegative) measure $M^k$ on $((0, \infty), \mathcal{B}((0,\infty)))$ such that
\begin{align}\label{eq:3.0}
M^k((s,\infty)) = k(s),\quad s\in (0, \infty).
\end{align}
By Fubini's theorem it is easy to show that
\begin{align}\label{eq:3.1}
\int_{(0,\infty)} \tau \wedge 1\ M^k(\,\mathrm{d}\tau) < \infty.
\end{align}
Define
\begin{align*}
\mathbb{C}_{\geq0}:=\{z\in\mathbb{C}\vert \operatorname{Re\,}z\geq0\}
\end{align*}
and
\begin{align*}
\mathbb{C}_{>0}:=\{z\in\mathbb{C}\vert\operatorname{Re\,}z>0\}.
\end{align*}
We define the following function $\psi^k:\mathbb{C}_{\geq0}\longrightarrow\mathbb{C}$ by
\begin{align}\label{eq:3.2}
\psi^k(\lambda):=\int_{(0,\infty)}(1-e^{-\lambda\tau})\ M^k(\,\mathrm{d}\tau),\quad \lambda\in\mathbb{C}_{\geq0},
\end{align}
which by \eqref{eq:3.1} is well-defined and holomorphic on $\mathbb{C}_{>0}$, as well as continuous on $\mathbb{C}_{\geq0}$ (see \cite[p.25]{SSV12} for details). Hence the same is true for the function
\begin{align}\label{eq:3.4'}
\lambda\mapsto e^{-t\psi^k(\lambda)},\quad \lambda\in\mathbb{C}_{\geq0},
\end{align}
for every $t\in [0,\infty)$. Furthermore for every $t\in[0,\infty)$, since $t\psi^k$ restricted to $(0,\infty)$ is a nonnegative Bernstein function (see \cite[Theorem 3.2]{SSV12}), there exists a unique probability measure $\mu_t^k$ on $([0,\infty),\mathcal{B}([0,\infty))$ such that
\begin{align}\label{eq:3.4}
\int_{[0,\infty)}e^{-\lambda s}\mu_t^k(\,\mathrm{d}s)=e^{-t\psi^k(\lambda)},\quad \lambda\in(0,\infty),
\end{align}
(see \cite[Theorems 3.7 and 1.4]{SSV12}). Furthermore, since the Laplace transform
\begin{align}
\mathcal{L}\mu_t^k(\lambda):=\int_{[0,\infty)}e^{-\lambda s}\mu_t^k(\, \mathrm{d}s)
\end{align}
is defined for all $\lambda\in\mathbb{C}_{\geq0}$ and is obviously holomorphic on $\mathbb{C}_{>0}$, as well as continuous on $\mathbb{C}_{\geq0}$, \eqref{eq:3.4} implies that
\begin{align}\label{eq:3.6}
\mathcal{L}\mu_t^k(\lambda)=e^{-t\psi^k(\lambda)} \quad \text{for all } \lambda\in\mathbb{C}_{\geq0}.
\end{align}
 In particular, we have for every $t\in[0,\infty)$ for the Fourier transform $\hat{\mu}^k_t$ of $\mu^k_t$
 \begin{align}\label{eq:3.7}
 \hat{\mu}_t^k(\lambda):&=\int_{(0,\infty)}e^{i\lambda s}\mu_t^k(\, \mathrm{d}s)=e^{-t\psi^k(-i\lambda)}= e^{-t\int_{(0,\infty)}(1-e^{i\lambda\tau})M^k(\, \mathrm{d}\tau)},\quad \lambda\in\mathbb{R}.
 \end{align}
 By \eqref{eq:3.1} the function $\R\ni\lambda\mapsto |\psi^k(-i\lambda)|$ is of at most linear growth.

 Now let us consider the $C_0$-semigroup $(U_t^k)_{t\geq0}$ on $\mathcal{H}$ with infinitesimal generator $(\Lambda^k, D(\Lambda^k,\mathcal{H}))$ introduced in Section \ref{section:2}. First we characterize this generator through its Fourier transform and as a corollary we prove that it coincides with $\partial^{*k}_t$ on an operator core.
\begin{Proposition}\label{proposition:3.1}
The generator $(\Lambda^k, D (\Lambda^k, \mathcal{H}))$ of $(U_t^{k})_{t > 0}$ (on $\mathcal{H}$), defined in Section \ref{section:2}, is given as follows
 $$D (\Lambda^k, \mathcal{H}) = \{ u \in \mathcal{H} \mid r \mapsto |\psi^k(-ir)|  \hat{u}(r) \in L^2 (\R; H_\mathbb{C})\},$$
 $$(\Lambda^k u)^\wedge(r) = -\psi^k(-ir) \hat{u} (r), \ r \in \R,$$
 where $H_\mathbb{C}$ denotes the complexification of $H$ and $\hat{u}$ denotes the Fourier transform of $u$ considered as a function from $\R$ to $H$, i.e. $u:=0$ on $(-\infty,0)$ and
 \begin{align*}
 \hat{u}(r):=\int_\mathbb{R}e^{irs}u(s)\,\mathrm{d}s,\quad r\in\mathbb{R}.
 \end{align*}
\end{Proposition}

\begin{proof}
 Below we consider each $\mu^{k}_t (ds)$ as a measure on all of $\R$, by defining
 \begin{align*}
  \mu_t^{k} (A) := \mu_t^{k} (A \cap [0, \infty)), \ A \in \mathcal{B} (\R).
 \end{align*}
  Let $D := \{ u \in \mathcal{H} \mid r \mapsto |\psi^k(-ir)| \hat{u} (r) \in L^2 (\R ; H_\mathbb{C})\}$.
 Then  for $u \in D$, because $u*\mu_t^k\in L^2(\R;H)$ and $\hat\mu_t^k$ is bounded, we have
 \begin{align*}
   \frac{1}{t} (U_t^{k} u - u)^\wedge (r)  & = \frac{1}{t} \Big( \int_{0}^{\infty} U_s u \ \mu_t^{k} (ds) - u \Big)^\wedge (r)\\
   & = \frac{1}{t} \Big( \int_{\R} u (\cdot - s) \ \mu_t^{k} (ds) - u \Big)^\wedge (r)\\
   & = \frac{1}{t} ( u * \mu_t^{k} - u)^\wedge (r)\\
   & = \frac{1}{t} \hat{u} (r) \Big( e^{-t \psi^k(-ir)} - 1 \Big) \xrightarrow[t \to 0]{} - \psi^k(-ir) \hat{u} (r)
  \end{align*}
 for $\mathrm{d}r$-a.e.~$r \in \R$. But since  for all $r \in \R$ and $t>0$,
 \begin{align*}
  \frac{1}{t} \Big| e^{-t \psi^k(-ir)} - 1 \Big|  \leq 2 \vert\psi^k(-ir)\vert,
 \end{align*}
 the last convergence also holds in $L^2(\mathbb{R};H_\mathbb{C})$. Hence $D \subset D (\Lambda^k, \mathcal{H})$ and
 \begin{equation}\label{eq:3.8}
  (\Lambda^k u)^\wedge (r) = - \psi^k(-ir) \hat{u} (r), \ r \in \R.
\end{equation}
 Because $\hat\mu_t^k$ is bounded, one similarly checks that
 \begin{equation}\label{eq:3.9}
  U_t^{k} D \subseteq  D  \quad \forall t > 0,
 \end{equation}
 and that $(\Lambda^k, D)$ is closed as an operator from $\mathcal{H}$ to $\mathcal{H}$.
 Since the function $\R\ni r\mapsto|\psi^k(-ir)|$ is at most of linear growth, $D$ is dense in $\mathcal{H}$. Hence \eqref{eq:3.8} implies (see \cite[Theorem X.49]{RS75})
 that $D$ is an operator core of $(\Lambda^k, D (\Lambda^k, \mathcal{H}))$, i.e.~ $D$ is dense in $D(\Lambda^k,\mathcal{H})$ with respect to the graph norm given by $\Lambda^k$. Consequently,
 $D = D (\Lambda^k, \mathcal{H})$ and $\Lambda^k$ is given by \eqref{eq:3.8}.
\end{proof}
\begin{Proposition}\label{corollary:3.2}
 $D(\Lambda^k,\mathcal{H}) \subset \mathcal F^k$ and for all $u \in D(\Lambda^k,\mathcal{H})$
\begin{equation}\label{eq:3.10}
  \Lambda^k u = - \frac{d}{dt} (k * u).
\end{equation}
\end{Proposition}
\begin{proof}
First  let $u \in D_0 := D(\Lambda^k, \mathcal{H}) \cap \mathcal{V} \cap D(\partial_t^{*k}) \cap L^\infty([0, \infty);H)$.
Then for $T \in (0, \infty)$
\begin{equation}\label{eq:3.11}
\| (k * u) (t) \|_{V^*} \leq \esssup\limits_{s \in [0,T]}
\|u(s)\|_{V^*}\int_0^tk(s)\,\mathrm{d}s \quad \text{ for $dt$-a.e.~} t \in [0, T],
\end{equation}
and the same inequality holds with $\|\cdot\|_H$ replacing $\|\cdot\|_{V^*}$.

Again we consider all appearing functions, originally only defined on $[0, \infty)$, as functions on all of $\R$
by defining them to be equal to zero on $\R \setminus [0, \infty)$.
As in the proof of Proposition \ref{proposition:3.1}, one can check that $D_0$ is dense in $\mathcal H$ and also that $U_t^k (D_0) \subset D_0$.
Concerning the latter we note that all spaces in the intersection defining $D_0$ are obviously invariant under $U_t^k$ except for $D(\partial^{*k}_t)$.
To see that this is also true for the latter, let $u\in D(\partial^{*k}_t)$. Then $U_t^ku=u*\mu_t^k\in L^1([0,\infty);V^*)$ and for $T\in(0,\infty)$ there exist $h\in L^1([0,T];V^*)$ and $v\in V^*$ such that for $r\in [0,T]$
\begin{align*}\
(k*u)(r)= v + \int_0^rh(\tau)\, \mathrm{d}\tau.
\end{align*}
But again by setting $h\equiv0$ on $(-\infty,0)$ and using Fubini's theorem
\begin{align*}
(k*U_t^ku)(r)&=(k*u*\mu_t^k)(r)
\\ &= v + \int_0^\infty\int_0^{r-s}h(\tau)\, \mathrm{d}\tau\, \mu_t^k( \mathrm{d}s)
\\ &= v + \int_0^\infty\int_{-s}^{r-s}h(\tau)\, \mathrm{d}\tau\, \mu_t^k(\mathrm{d}s)
\\ &= v + \int_0^\infty\int_0^rh(\tau-s)\, \mathrm{d}\tau\, \mu_t^k(\mathrm{d}s)
\\ &= v + \int_0^r(h*\mu_t^k)(\tau)\, \mathrm{d}\tau,\quad r\in[0,T].
\end{align*}
Since $h*\mu_t^k\in L^1([0,T];V^*)$, this implies that $U_t^ku\in D(\partial_t^{*k})$.
Again applying Theorem X.49 from \cite{RS75} we obtain that $D_0$ is an operator core of $(\Lambda^k, D(\Lambda^k, \mathcal H))$.
Hence it remains to prove \eqref{eq:3.10}.

Let us start with calculating the Laplace transform $\mathcal L$ of the right hand side of \eqref{eq:3.10} for any $u \in D(\partial_t^{*k}) \cap L^\infty([0, \infty);V^*)$.
So let $\lambda \in (0, \infty)$.
Then integrating by parts, using \eqref{eq:3.11} and Fubini's Theorem we obtain
\begin{align*}
\mathcal L \left( \frac{\mathrm{d}}{\mathrm{d}t} (k * u) \right) (\lambda) &= \int_0^\infty e^{-\lambda t} \frac{\mathrm{d}}{\mathrm{d}t} (k * u) (t) \, \mathrm{d}t
\\ &= \lim\limits_{T \rightarrow \infty} \left( e^{-\lambda T} (k * u) (T) + \lambda \int_0^T e^{-\lambda t} (k * u) (t) \, \mathrm{d}t \right)
\\ &= \lambda \int_0^\infty e^{-\lambda t} \int_0^t k(t-s) u(s) \, \mathrm{d}s \, \mathrm{d}t
\\ &= \lambda \int_0^\infty \int_s^\infty k(t-s) e^{-\lambda t} \, \mathrm{d}t \, u(s) \,\mathrm{d}s
\\ &= \int_0^\infty e^{-\lambda s} u(s) \,\mathrm{d}s \, \lambda \int_0^\infty M^k((t,\infty)) e^{-\lambda t} \, \mathrm{d}t
\\ &= \lambda\int_0^\infty\int_{(t,\infty)} M^k(\,\mathrm{d}s)e^{-\lambda t}\, \mathrm{d}t \, \mathcal{L}u(\lambda)
\\ &= \int_{(0, \infty)} \int_0^s\lambda e^{-\lambda t}\, \mathrm{d}t \, M^k(\mathrm{d}s)\, \mathcal{L}u(\lambda)
\\ &= \psi^k(\lambda)\mathcal{L}u(\lambda),
\end{align*}
where we used \eqref{eq:3.0} in the fifth inequality and \eqref{eq:3.2} in the last inequality.\\
For the left-hand side of \eqref{eq:3.10} and $u\in D_0$ we find for all $h \in H$, $\lambda \in (0, \infty)$, because of \eqref{eq:3.6}
\begin{align*}
\left\langle \int_{0}^{\infty} \Lambda^k u (r) e^{- \lambda r} \, \mathrm{d}r, h \right\rangle_H &= \lim \limits_{t \to 0} \frac{1}{t} \int_{0}^{\infty} \left\langle U^{k}_t  u(r) - u(r), h \right\rangle_H e^{- \lambda r} \, \mathrm{d}r \\
 & = \lim\limits_{t \to 0} \frac{1}{t} \Big(\mathcal{L} ( \langle u,h\rangle_H * \mu_t^{k}) - \mathcal{L} ( \langle u,h\rangle_H) \Big) (\lambda) \\
 & = \lim\limits_{t \to 0} \frac{1}{t} (e^{-t \psi^k(\lambda)} -1) \, \mathcal{L} ( \langle u,h\rangle_H) (\lambda)\\
 & = - \psi^{k}(\lambda) \left\langle \mathcal{L} u (\lambda), h \right\rangle_H.
\end{align*}
Hence, $\mathcal{L} (\Lambda^k (u)) (\lambda) = - \psi^{k}(\lambda) \mathcal{L} u (\lambda)$ and \eqref{eq:3.10} follows  for $u \in D_0$.

Now let $u \in D(\Lambda^k,\mathcal H)$.
Then, since $D_0$ is an operator core for $(\Lambda^k, D(\Lambda^k,\mathcal H))$, there exist $u_n \in D_0$, $n \in \mathbb N$, such that as $n \rightarrow \infty$
\begin{align*}
	u_n \longrightarrow u \quad \text{ and } \quad \Lambda^k u_n \longrightarrow \Lambda^k u \; \text{ in } \mathcal H.
\end{align*}
Let $T \in (0, \infty)$.
Then as $n \longrightarrow \infty$ by \eqref{eq:2.7'} and, since \eqref{eq:3.10} holds for $u_n$,
\begin{align*}
	k * u_n \longrightarrow k * u
\end{align*}
and
\begin{align*}
	-\frac{\partial}{\partial t} (k * u_n) \longrightarrow \Lambda^k u
\end{align*}
in $L^1([0,T];V^*)$.
Hence, the last assertion follows by the completeness of $W^{1,1}([0,T];V^*)$.

\end{proof}

After these preparations we can prove the first part of Theorem \ref{thm:2.1}.

\begin{proof}[Proof of Theorem \ref{thm:2.1}(i).]

 Let $u \in \mathcal{F}^k$.
 Then there exist $u_n \in \mathcal V \cap D(\Lambda^k,\mathcal H)$, $n \in \mathbb N$, such that as $n \to \infty$
\begin{align}\label{eq:3.12}
   u_n &\longrightarrow u \text{ in } \mathcal{V} \quad \text{ and } \quad  - \partial_t^{*k} u_n = \Lambda^k u_n \longrightarrow \Lambda^k u \text{ in } \mathcal{V}^*,
\end{align}
 where we used Proposition \ref{corollary:3.2}.
Let $T\in (0, \infty)$.  By \eqref{eq:2.7'}, $k * u_n \longrightarrow k * u$ in $L^{\alpha} ([0, T]; V)$, hence in $L^1([0,T];V^*)$, as $ n \to \infty$ and for $p:=\operatorname{min\,}\{2,\frac{\alpha}{\alpha-1}\}$ the latter part of \eqref{eq:3.12} implies that $\partial_t^{*k}u_n,\ n\in\N$, are bounded in $L^p([0,T];V^*)$. Hence the Cesaro mean of a subsequence of $(\partial_t^{*k}u_n)_{n\in\N}$ converges strongly in $L^p([0,T];V^*)$, hence in $L^1([0,T];V^*)$.
Therefore, by completeness $k * u \in W^{1,1} ((0, T); V^*)$ and
\begin{align*}
   \Lambda^k u = -\frac{d}{dt} (k * u) \ \text{on} \  [0, T]\ \mathrm{d}t\text{-a.e.~}
\end{align*}
The last part of the assertion then follows by \cite[Theorem 1.19, pp.25]{Ba10}.
\end{proof}
To prove Theorem \ref{thm:2.1}(ii) we need some preparations.
\begin{Lemma}\label{lemma:3.3}
Let $\gamma\in(0,\infty)$. Then for all $u\in\mathcal{H}$, $t\geq0$,
\begin{align*}
\int_0^\infty \| U_t^{k}u(s)\|_H^2e^{-\gamma s}\, \mathrm{d}s\leq e^{-\psi^k(\gamma)t}\int_0^\infty\| u(s)\|_H^2e^{-\gamma s}\, \mathrm{d}s.
\end{align*}
\end{Lemma}
\begin{proof}
Let $u\in\mathcal{H}$, $t\geq0$. Then
\begin{align*}
 \int_0^\infty\|U_t^{k}u(s)\|_H^2 \, e^{-\gamma s}\,\mathrm{d}s = \int_0^\infty\|(u*\mu_t^k)(s)\|_H^2 \, e^{-\gamma s}\,\mathrm{d}s \leq \mathcal{L}(\|u\|^2_H)(\gamma) \; e^{-t\psi^k(\gamma)},
\end{align*}
where we used Jensen's inequality and \eqref{eq:3.6} in the last step.
\end{proof}

\begin{Lemma}\label{lemma:3.4}
Let $\gamma\in(0,\infty)$ and $u\in D(\Lambda^k,\mathcal{H})$. Then
\begin{align*}
\int_0^\infty\langle\Lambda^ku(s),u(s)\rangle_He^{-\gamma s}\, \mathrm{d}s \leq -\frac12\psi^k(\gamma)\int_0^\infty \|u(s)\|_H^2e^{-\gamma s}\, \mathrm{d}t.
\end{align*}
\end{Lemma}
\begin{proof}
Since
\begin{align*}
\Lambda^ku=\lim_{\varepsilon\rightarrow0}\frac1\varepsilon(U_\varepsilon^ku-u)\quad \text{in $\mathcal{H}$, hence in $L^2_\gamma([0,\infty);H)$,}
\end{align*}
we have by Lemma \ref{lemma:3.3} and the Cauchy--Schwarz inequality
\begin{align*}
&~~ \int_0^\infty\langle \Lambda^k u(s),u(s)\rangle_He^{-\gamma s}\, \mathrm{d}s \\
&= \lim_{\varepsilon\rightarrow 0} \frac1\varepsilon\left[\int_0^\infty\langle U_\varepsilon^ku(s),u(s)\rangle_He^{-\gamma s}\, \mathrm{d}s -\int_0^\infty\langle u(s),u(s)\rangle_He^{-\gamma s}\,\mathrm{d}s \right]
\\ &\leq \lim_{\varepsilon\rightarrow 0}\frac1\varepsilon\left[e^{-\frac\varepsilon 2 \psi^k(\gamma)}-1 \right]\int_0^\infty\|u(s)\|_H^2 e^{-\gamma s}\, \mathrm{d}s
\\ &= -\frac12\psi^k(\gamma)\int_0^\infty\|u(s)\|^2_He^{-\gamma s}\,\mathrm{d}s.
\end{align*}
\end{proof}
Now we can prove the second part of Theorem \ref{thm:2.1}.
\begin{proof}[Proof of Theorem \ref{thm:2.1}(ii).]
Let $u\in\mathcal{F}^k$. By definition of $(\Lambda^k,\mathcal{F}^k)$ there exist $u_n\in D(\Lambda^k,\mathcal{H})\cap\mathcal{V}$ such that as $n\longrightarrow\infty$
\begin{align*}
u_n\longrightarrow u \quad \text{in $\mathcal{V}$} \quad \text{ and } \quad \Lambda^k u_n\longrightarrow\Lambda^ku\quad \text{in $\mathcal{V}^*$.}
\end{align*}
Hence by Lemma \ref{lemma:3.4}
\begin{align*}
\int_0^\infty {}_{V^*}\langle\Lambda^ku(s),u(s)\rangle_Ve^{-\gamma s}\, \mathrm{d}s
& = \lim_{n\rightarrow\infty} \int_0^\infty {}_{V^*}\langle\Lambda^ku_n(s),u_n(s)\rangle_Ve^{-\gamma s}\, \mathrm{d}s
\\& \leq \lim_{n\rightarrow\infty} -\frac12\psi^k(\gamma)\int_0^\infty\|u_n(s)\|^2_He^{-\gamma s}\,\mathrm{d}s
\\& = -\frac12\psi^k(\gamma)\int_0^\infty\|u(s)\|_H^2e^{-\gamma s}\,\mathrm{d}s,
\end{align*}
since $u_n\longrightarrow u$ in $\mathcal{V}$ as $n\rightarrow\infty$, implies that $u_n\longrightarrow u$ in $\mathcal{H}$, hence in $L^2_\gamma([0,\infty);H)$ as $n\rightarrow\infty$. Hence the assertion follows by Theorem \ref{thm:2.1}(i).
\end{proof}

\section{Proof of main existence and uniqueness result: the deterministic case}\label{section:4}

In this section we proof Theorem \ref{thm:2.2}, so assume that \ref{cond:k} and \ref{cond:H1}--\ref{cond:H4} hold.

As in \cite{LRS18} the proof heavily relies on a general perturbation result of operators $\mathcal{A}$ of the type as in Theorem \ref{thm:2.2}, which we briefly recall now.

As in \cite{St99} we consider a generator $\Lambda$, with domain $D(\Lambda, \mathcal H)$, of a $C_0$-contraction semigroup of linear operators on $\mathcal H$ whose restrictions to $\mathcal V$ form a $C_0$-semigroup of linear operators on $\mathcal V$.
The generator of the latter is again $\Lambda$, but with domain $D(\Lambda, \mathcal V) := \{ u \in \mathcal V \cap D(\Lambda, \mathcal H)  \mid \Lambda u \in \mathcal V\}$.
Then $D(\Lambda, \mathcal V)$ is dense in $\mathcal V$, hence so is $D(\Lambda, \mathcal H) \cap \mathcal V$.
By \cite[Lemma 2.3]{St99}, $\Lambda :\mathcal{D }(\Lambda, \mathcal{H}) \cap \mathcal{V} \longrightarrow \mathcal{V}^*$ is closable as an operator from $\mathcal{V}$ to $\mathcal{V}^*$. Denoting its closure by $(\Lambda, \mathcal{F})$ we obtain that $\mathcal{F}$ is a Banach space with norm
\begin{align*}
  {\| u\|}_\mathcal{F} := {(\| u \| ^2_\mathcal{V} +  \| {\Lambda u\| ^2_\mathcal{ V^*})}}^\frac{1}{2}, u\in \mathcal{F}.
 \end{align*}
Now we can formulate the following perturbation result.
\begin{Theorem}\label{thm:4.2}
 Let conditions \ref{cond:H1}--\ref{cond:H4} hold. Assume that in \ref{cond:H2} we have $C_1 = 0$. Then for every $f\in \mathcal{V}^*$ there exists $u \in \mathcal{F}$ such that $\mathcal{A}u - \Lambda u = f$.
\end{Theorem}
This result is a generalization of \cite[Proposition 3.2]{St99}. We replace the strong monotonicity assumption in \cite[Proposition 3.2]{St99} by the classical monotonicity, i.e.~ \ref{cond:H2} with $C_1=0$, and consider a reflexive Banach space $\mathcal V$, while this space was assumed to be a Hilbert space in \cite{St99}. A rather concise proof in this more general case was given in \cite{LRS18}. Since this result is crucial for Theorem \ref{thm:2.2} and for the convenience of the reader we include a more detailed proof in the Appendix of this paper. Now we are prepared to prove the second main result of this paper.

\begin{proof}[Proof of Theorem \ref{thm:2.2}(i)] Existence: \\
 \underline{Case 1:} $u_0 = 0$.\\
 Consider the operator
 \begin{align*}
 \tilde{A} := A + C_1I,
 \end{align*}
where $I\colon V \longrightarrow V^*$, $I(u):=u, u \in V$, and let $\tilde {\mathcal{A}}$ be defined as $\mathcal{A}$ was for $A$. Then we can apply Theorem \ref{thm:4.2} with $\mathcal{A}$ replaced by $\tilde{\mathcal{A}}$ and $\Lambda$ replaced by $\Lambda^k=-\partial^{*k}_t$ with domain $\mathcal{F}^k$ (see Theorem \ref{thm:2.1}(i)). Hence for every $g\in \mathcal{H}(\subset\mathcal{V}^*)$ there exists $u_g\in \mathcal{F}^k$ such that
\begin{equation}\label{eq:4.1}
 \partial^{*k}_t u_g + \mathcal{A} u_g + C_1 u_g = g + f \quad\text{ in } \mathcal{V^*}.
\end{equation}
Define: $\mathcal{H}_T:=L^2([0,T];H)$. Then $\mathcal{H}_T\hookrightarrow \mathcal{H} $ by the map $i(g)=\begin{cases}g\text{ on }[0,T]\\ 0\text{ on }(T, \infty)\end{cases}$.\\
Consider the map $\mathcal{H} \supset \mathcal{H}_T \ni g \mapsto C_1 u_{g\upharpoonright[0,T]} \in \mathcal{H}_T $, where for a function $h \colon [0, \infty) \longrightarrow H$ we denote its restriction to $[0,T]$ by $h_{\upharpoonright[0,T]}$.

By \ref{cond:H2} and Theorem \ref{thm:2.1}(ii) we have for all $\gamma \in (0, \infty);\ g_1, g_2 \in \mathcal{H}_T$
\begin{align*}
&~~~~ \frac{1}{2}\psi^k(\gamma) \int^T_0  \| C_1 u_{g_1} (s)- C_1u_{g_2} (s) \|^2_H e^{-\gamma s}\,\mathrm{d}s
 \\ & \leq C^2_1 \int^\infty_0 {}_V\langle u_{g_1} (s) - u_{g_2}(s), \partial^{*k}_t (u_{g_1} - u_{g_2}) + A(s, u_{g_1} (s)) - A(s, u_{g_2}(s))
 \\ &\hspace{207pt} + C_1 (u_{g_1} (s) - u_{g_2} (s)) \rangle_{V^*} e^{-\gamma s}\,\mathrm{d} s\\
  &=C_1\int^T_0 {\langle C_1 u_{g_1} (s) - C_1 u_{g_2} (s), g_1(s)-g_2(s) \rangle}_H e^{-\gamma s} \,\mathrm{d}s.
 \end{align*}
Hence by the Cauchy--Schwarz inequality
\begin{equation*}
 {\left(\int^T_0 \| C_1 u_{g_1} (s) - C_1 u_{g_2} (s) \|^2_H e^{-\gamma s}\,\mathrm{d}s\right)}^\frac{1}{2} \leq \frac{2C_1}{\psi^k (\gamma)} {\left(\int^T_0 \| g_1 (s) - g_2 (s)\|^2_H e ^{-\gamma s}\,\mathrm{d}s\right)}^\frac{1}{2}.
\end{equation*}
We recall that by assumption
\begin{align*}
 \frac{2C_1}{\psi^k (\gamma)}< 1,
\end{align*}
which can always be achieved for large enough $\gamma$ by \eqref{eq:2.9}, if $M^k((0,\infty)) = \infty$, i.e.~if $\lim\limits_{s \to 0} k(s) = \infty$.
Hence by Banach's fixed point theorem there exists $g \in \mathcal{H}_T$
\begin{align*}
C_1 u_g = g\quad \mathrm{d}t\text{-a.e.~on }[0, T].
\end{align*}
But then by \eqref{eq:4.1}
\begin{align*}
\partial^{*k}_t u_g(t) + A(t, u_g (t)) = f(t)\quad \text{for }\mathrm{d}t\text{-a.e.~} t \in [0,T],
\end{align*}
so \eqref{eq:2.1} holds for $u_0 = 0$. Furthermore by construction $u_g \in \mathcal{F}^k$. In particular, \eqref{eq:2.10} holds for $u_0=0$.\\
\underline{Case 2:} $u_0 \in V$.\\
Let $\varphi$ be as in the assertion of the Theorem. Set $x:=u_0$ and define $\mathcal{A}_x$ as $\mathcal{A}$, but with
\begin{align*}
 A_x (t, v) := A (t, v + x \varphi (t)), \ t > 0, \ v \in V,
\end{align*}
replacing $A$. Then by Case 1 there exist $u_x \in \mathcal{F}^k$ such that
\begin{align*}
 \frac{d}{dt} (k * u_x(t)) + A_x(t,u_x(t)) = f(t)\quad \text{for } \mathrm{d}t\text{-a.e.~} t\in[0,T].
\end{align*}
Define $u := u_x + x \varphi$. Then $u - x \varphi (= u_x)$ satisfies \eqref{eq:2.10} and
\begin{align*}
 \partial^{*k}_t (u(t) - x) + A(t,u(t)) = f(t) \quad dt \text{-a.e.~on }  (0, T)
\end{align*}
and \eqref{eq:2.1} is solved.

The last part of assertion (i) of Theorem \ref{thm:2.2} follows by the last part of Theorem \ref{thm:2.1}(i).\\
Uniqueness: Let $u_1, u_2$ be two solutions of \eqref{eq:2.1} on $[0, T]$ such that $u_1- \varphi u_0, u_2 - \varphi u_0 \in \mathcal{F}^k$ with $\varphi$ as in the assertion. Then $u_1- u_2 \in \mathcal{F}^k$ and by Theorem \ref{thm:2.1}(ii)
\begin{align*}
 0 &= \int^\infty_0 {}_V \langle u_1(s) - u_2(s), \partial^{*k}_t (u_1(s) - u_2(s)) + A(s,u_1 (s))-A(s,u_2(s)\rangle_{V^*} e^{-\gamma s}\, \mathrm{d}s\\
 &\geq \left(\frac{1}{2} \psi^k(\gamma) - C_1\right) \int^\infty_0 \| u_1(s)-u_2(s)\|^2_{H} e^{-\gamma s}\,\mathrm{d}s\\
 &\geq 0,
\end{align*}
since by assumption $\psi^k(\gamma) > 2C_1$. Hence $u_1 = u_2$.\\
\end{proof}
\begin{proof}[Proof of Theorem \ref{thm:2.2}(ii).] That under assumption \ref{cond:k2} equation \eqref{eq:2.1} can be rewritten as \eqref{eq:2.11} was already explained in  Section \ref{section:2}  of this paper. The last part of the assertion is an elementary fact about convolutions in Lebesgue $L^p$-spaces.
\end{proof}

\section{Proof of the stochastic case}\label{section:5}
The proof of Theorem \ref{thm:2.3} follows from Theorem \ref{thm:2.2} by a simple shift argument (cf.~\cite{LRS18}).

\begin{proof}[Proof of Theorem \ref{thm:2.3}]
Let $u(t)=X(t)-F(t)$, then $u(t)$ satisfies the following equation
\begin{equation}\label{eq:5.1}
 \partial_t^{*k_1}(u(t)-x)+A(t,u(t)+F(t))=0,  \quad 0<t<T.
\end{equation}
Define
$$ \tilde{A}(t,u)=A(t,u+F(t)), u\in V. $$
Since $F\in  V$ $\mathrm{d}t\otimes\mathbb{P}$-a.e., it is easy to see that $\tilde{A}$ still satisfies \ref{cond:H1}--\ref{cond:H4}.
Hence assertion $(i)$ follows by Theorem \ref{thm:2.2}$(i)$. Assertion $(ii)$ is then proved analogously to Theorem \ref{thm:2.2}$(ii)$.

The $(\mathcal{F}_t)$-adaptedness of the solution follows by the proofs of Theorem \ref{thm:4.2} and Lemma \ref{lemma:a.2}.
The last two assertions are obvious.
\end{proof}

\section{Examples of Kernels}\label{section:6}
In this section we give some examples of kernels $k$ which satisfy both condition \ref{cond:k} and \ref{cond:k2} needed to apply Theorems \ref{thm:2.1}, \ref{thm:2.2}
and \ref{thm:2.3} in Section \ref{section:2}.
\begin{Example}[Fractional Caputo derivative]
\label{exa:Caputo-derivative}Let $0<\beta<1$ be given and define
the function $k$ on $[0,\infty)$ by
\[
k(t):=g_{1-\beta}(t)=\frac{t^{-\beta}}{\Gamma(1-\beta)},\quad t\in[0,\infty).
\]
Then $k$ is nonnegative, nonincreasing function on $[0,\infty)$
and we have $\lim_{t\to0}k(t)=\infty$ and $\lim_{t\to \infty }k(t)=0$.
It is well known that $\partial_{t}^{*k}(f-f(0))$ corresponds to
the Caputo derivative of $f$ and the problem stated in \eqref{eq:1.1}
has been treated in \cite{LRS18}. The associated L{\'e}vy measure
is absolutely continuous with respect to the Lebesgue measure and
is given by
\begin{equation}
M_{\beta}^{k}(\mathrm{d}t)=\frac{\beta}{\Gamma(1-\beta)}t^{-(1+\beta)}\,\mathrm{d}t.\label{eq:Levy-mesaure1}
\end{equation}
It is simple to verify that $k$ satisfies (k) and the corresponding
$\tilde{k}\in L_{\mathrm{loc}}^{1}([0,\infty))$ is given by
\[
\tilde{k}(t)=\frac{t^{\beta-1}}{\Gamma(\beta)},\quad t\in[0,\infty).
\]
Hence condition ($\tilde{\textrm{k}}$) is satisfied. The pair $(k,\tilde{k})$
is called Sonine kernels and $(\tilde{k}*k)(t)=1$, $t\in[0,\infty)$
is known as Sonine condition, see \cite{Sonine1884} and \cite{Samko2003}
for a survey.
\end{Example}

\begin{Example}[Truncated $\beta$-stable subordinator, cf.~Example 2.1-(ii) in \cite{Che17}]
A process $S(t)$, $t\ge0$ is called truncated $\beta$-stable subordinator
if it is driftless and its L{\'e}vy measure is
\[
M_{\delta}^{k}(\mathrm{d}x):=\frac{\beta}{\Gamma(1-\beta)}x^{-(1+\beta)}1\!\!1_{(0,\delta]}(x)\,\mathrm{d}x,\qquad\delta>0.
\]
The kernel $k$ defined by
\[
k(t):=M_{\delta}^{k}((t,\infty))=\frac{\beta}{\Gamma(1-\beta)}1\!\!1_{(0,\delta]}(t)\int_{t}^{\delta}x^{-(1+\beta)}\,\mathrm{d}x=\frac{1\!\!1_{(0,\delta]}(t)}{\Gamma(1-\beta)}(t^{-\beta}-\delta^{-\beta})
\]
induces the following generalized time-fractional derivative
\[
\partial_{t}^{*k}(f-f(0))(t)=\frac{1}{\Gamma(1-\beta)}\frac{d}{dt}\int_{(t-\delta)^{+}}^{t}\big((t-s)^{-\beta}-\delta^{-\beta}\big)(f(s)-f(0))\,\mathrm{d}s.
\]
Here for $a\in\mathbb{R}$, $a^{+}:=\max\{a,0\}$. This is the generalized
time-fractional derivative whose value at time $t$ depends only on
the $\delta$-range of the past of $f$ in contrast to the usual case
which depends on the history of $f$ on $(0,t)$. Notice that $\lim_{\delta\to0} k(t)=\frac{1}{\Gamma(1-\beta)}t^{-\beta}$. We have $\lim_{t\to0} k(t)=\infty$
and $\lim_{t\to\infty} k(t)=0$. Hence, $k$ satisfies condition
\ref{cond:k}, but also \ref{cond:k2}, because $M_{\delta}^{k}$ is absolutely continuous with respect to
the Lebesgue measure. Hence the existence of the kernel $\tilde{k}$ follows from the theory of complete Bernstein
functions, see Theorem 6.2
in \cite{SSV12}.
\end{Example}

\begin{Example}[Distributed order derivative]
\label{exa:distr-order-deriv}Let $g_{\beta}$ as in (\ref{eq:1.3})
and define the kernel $k$ by
\[
k(t):=\int_{0}^{1}g_{\beta}(t)\,\mathrm{d}\beta,\quad t\ge0.
\]
The corresponding generalized time-fractional derivative is called
\emph{distributed order derivative} and it may be written as
\[
\partial_{t}^{*k}(f-f(0))(t)=\int_{0}^{1}\partial_{t}(k*(f-f(0))(t)\,\mathrm{d}\beta.
\]
The kernel $k$ is a nonincreasing, nonnegative function on $[0,\infty)$
which belongs to $L_{\mathrm{loc}}^{1}([0,\infty))$. Moreover, $\lim_{t\to0}k(t)=\infty$
and $\lim_{t\to\infty}k(t)=0$. The associated nonnegative kernel
$\tilde{k}$ such that $\tilde{k}*k=1$ has the form
\[
\tilde{k}(t)=\int_{0}^{\infty}\frac{e^{-st}}{1+s}\,\mathrm{d}s
\]
and we have $\tilde{k}\in L_{\mathrm{loc}}^{1}([0,\infty))$, so condition
($\tilde{\textrm{k}}$) is satisfied.
\end{Example}

\begin{Example}[Exponential weight]
\label{exa:power-exp}For any $\gamma\ge0$,  $\lambda>0$  and $0<\beta<1$ define
the kernel $k$ by
\[
k(t):=g_{1-\beta}(t)e^{-\lambda t}=\yx\frac{t^{-\beta}}{\Gamma(1-\beta)}\xy e^{-\lambda t}.
\]
The kernel $k$ is nonnegative, nonincreasing and $k\in L_{\mathrm{loc}}^{1}([0,\infty))$,
hence $k$ satisfies condition (k). We have $\lim_{t\to0}k(t)=\infty$
and $\lim_{t\to\infty}k(t)=0$. The associated nonnegative $\tilde{k}$
such that $\tilde{k}*k=1$ is given by
\[
\tilde{k}(t)=\gamma^{\beta}+\frac{\beta}{\Gamma(1-\beta)}\int_{t}^{\infty}\frac{e^{-\gamma s}}{s^{1+\beta}}\,\mathrm{d}s,\quad t\in[0,\infty).
\]
The fact that $\tilde{k}*k=1$ may be checked by applying the Laplace
transform to both sides of the equation. Moreover, a simple integration
shows that $\tilde{k}\in L_{\mathrm{loc}}^{1}([0,\infty))$, hence
condition \ref{cond:k2} is satisfied.
\end{Example}

\begin{Example}[Gamma subordinator]
\label{exa:gamma-subordinator}Let $a,b>0$ be given and $k$ the
kernel defined by
\[
k(t):=a\Gamma(0,bt),\;t\in[0,\infty),
\]
where $\Gamma(\nu,x):=\int_{x}^{\infty}t^{\nu-1}e^{-t}\,dt$ is the
upper incomplete gamma function. It follows from the properties of
$\Gamma(\nu,x)$ that $k$ is a locally integrable, nonnegative, nonincreasing
function on $[0,\infty)$ and we have $\lim_{t\to0}k(t)=\infty$ and
$\lim_{t\to\infty}k(t)=0$. Hence, $k$ satisfies condition (k). The
kernel $k$ is related to the gamma subordinator (see for example
\cite[Ch.~III]{Bertoin96}) through its Laplace transform, namely
the process with Laplace exponent equal to
\[
\lambda\int_{0}^{\infty}e^{-\lambda t}k(t)\,\mathrm{d}\tau=a\log\left(1+\frac{\lambda}{b}\right)=a\int_{0}^{\infty}(1-e^{-\lambda t})t^{-1}e^{-bt}\,\mathrm{d}t,\quad\lambda>0,
\]
where the second equality stems from the Frullani integral. Hence,
the L{\'e}vy measure is $M_{a,b}^{k}(\mathrm{d}t)=at^{-1}e^{-bt}\,\mathrm{d}t.$
The existence of a positive $\tilde{k}\in L_{\mathrm{loc}}^{1}([0,\infty))$
such that $\tilde{k}*k=1$ is a consequence of the fact that $M_{a,b}^{k}$
is absolutely continuous with respect to the Lebesgue measure and
the theory of complete Bernstein functions, see Theorem 6.2 in \cite{SSV12}. Hence condition \ref{cond:k2} is satisfied.
\end{Example}

\begin{Example}[Multi-term derivative]
Let $0<\beta<1$ and $0<\alpha<1$ be given. Define the kernel $k$
by
\[
k(t):=g_{1-\beta}(t)+g_{1-\alpha}(t),\quad t>0.
\]
The kernel $k$ is completely monotone, that is $k\in C^{\infty}((0,\infty))$
and $(-1)^{n}k^{(n)}(t)\ge0$ for all $t>0$ and $n\in\mathbb{N}\cup\{0\}$.
The corresponding generalized time-fractional derivative $\partial_{t}^{*k}$
is called \emph{multi-term fractional derivative}. We have $\lim_{t\to0}k(t)=\infty$
and $\lim_{t\to\infty}k(t)=0$. It follows from Example \ref{exa:Caputo-derivative}
that the L{\'e}vy measure $M_{\beta,\alpha}^{k}$ defining $k$ is
the sum of two L{\'e}vy measures of the type (\ref{eq:Levy-mesaure1}).
It follows from Theorem 5.5 and Corollary 5.6 of \cite{Gripenberg2019}
that there exists a nonnegative kernel $\tilde{k}\in L_{\mathrm{loc}}^{1}([0,\infty))$
such that $\tilde{k}*k=1$ and its Laplace transform is
\[
\mathcal{L}\tilde{k}(\lambda)=\frac{1}{\lambda^{\alpha}+\lambda^{\beta}}.
\]
Hence, the kernel $k$ satisfies both conditions \ref{cond:k} and \ref{cond:k2}. This example may be generalized
to kernels $k(t):=\sum_{j=1}^{n}a_{j}\frac{t^{-\beta_{j}}}{\Gamma(1-\beta_{j})}$
with $a_{j}>0$ and $0<\beta_{1}<\ldots<\beta_{n}<1$.
\end{Example}

\section{Applications to quasi-linear (S)PDE}\label{section:7}

In this section we apply  Theorems  \ref{thm:2.2} and \ref{thm:2.3}  to   (stochastic) generalized porous medium equations, (stochastic) generalized $p$-Laplace equations, and  (stochastic) generalized fast-diffusion equations (cf. \cite{BR15,LR15}) with time-fractional derivative. Here for simplicity we mainly concentrate on the deterministic case, the extension to the stochastic case is straightforward.

\subsection{Generalized porous medium equations}
We introduce the model as in \cite{RRW}. Let $(E,\mathcal{B},{\m})$ be a separable $\sigma$-finite measure space  and
$(L,\D(L))$ a negative definite self-adjoint linear operator on
$L^2({\bf m})$ having
 discrete spectrum.  Let $$
  (0<) \ll_1\le \ll_2\le \cdots
$$ be all eigenvalues of $-L$ including multiplicities with unit eigenfunctions $\{e_i\}_{i\ge 1}$.
Let $H$ be the dual space of the $\D((-L)^{\ff 12})$ with respect to $L^2(\m)$; i.e. $H$ is the  completion of $L^2(\m)$ under the inner
product
$$\<x,y\>:= \sum_{i=1}^\infty \ff 1 {\ll_i}
\m(xe_i)\m(ye_i),$$ where $\m(x):= \int_E x\d\m$ for $x\in L^1(\m).$    Let
$$\Psi, \Phi: [0,\infty)\times \R \to \R$$
be   measurable,  and be continuous in the second variable.
   We consider the following  generalized porous medium equation  with generalized time-fractional derivative
\begin{equation}\label{PME}
\partial_t^{*k} (X_t-x_0) = L\Psi(t,X_t)+ \Phi(t,X_t).\end{equation}

To verify conditions $(H1)$, $(H2)$, $(H3)$ and $(H4)$ for $A(t,v):= L \Psi(t,v)+ \Phi(t,v),$ we assume
that for a fixed constant $r\ge 1$,
\begin{equation}\label{7.1}\begin{split} & |\Psi(t,s)|+|\Phi(t,s)|\le c(1+|s|^r), \ \ s\in\R, t\ge 0,\\
& -\m\big((\Psi(t,x)-\Psi(t,y))( x-y)\big) +\m\big((\Phi(t,x)-\Phi(t,y)) (-L)^{-1} (x-y)\big) \\
&\quad \le
K\|x-y\|_H^2-\delta \|x-y\|_{r+1}^{r+1},\ \  t\ge 0,\end{split} \end{equation} hold for some constants $c, K,\delta>0$ and all $x,y\in L^{r+1}(\m),$
where $\|\cdot\|_{r+1}$ is the norm in $L^{1+r}(\m).$  Obviously, the assumptions above are satisfied provided
$\Psi(t,s)= h(t) |s|^{r-1}s$ and $\Phi(t,s)= g(t) s$, $t\in[0,T]$, $s\in\R$, with
$0<\inf h\le \sup h<\infty$ and $ \|g\|_\infty<\infty$.

~~

\noindent\textbf{Example  7.1}  Let $V=L^{1+r}(\m)$ and $V^*$ be the dual space of $V$ with respect to $H$. Then it is easy to see that \eqref{7.1} implies that
$(H1)$, $(H2)$, $(H3)$ and $(H4)$ hold for (see \cite[page 137]{RRW})
$$A(t,v):= L\Psi(t,v) +\Phi(t,v).$$   Therefore,  Theorem \ref{thm:2.2} is applicable to the  time-fractional  generalized porous medium equation (\ref{PME}) if $k$ satisfies \ref{cond:k}, \ref{cond:k2} respectively.
\begin{Remark} (i) Let $r>1$ and $\DD$ be the Dirichlet Laplacian on an open  domain $D\subset \R^d$. Let $L= \DD$ if $D$ is bounded and, in addition,  $r\le \frac{2d}{d+2}$, or  $L= -(-\DD)^\alpha$  for some constant $\alpha \in (0, \frac{d}{2}) \cap (0, 1]$ if $D= \R^d$ (the definition of $V$ and $H$ should be revised in the latter case, see \cite{RRW}). Let   $$\Phi(t,s)=cs,\ \  \Psi(t,s)=s|s|^{r-1}, $$ for some constant $c\in\R$ (see \cite[Example 3.4]{RRW} for possible more general cases).
 Then the assertions in Theorem \ref{thm:2.2} hold.

 (ii) Similarly, we could apply Theorem \ref{thm:2.3} to investigate the  time-fractional  stochastic
 generalized porous medium equation
 \begin{equation}\label{SPME}
\partial_t^{*k_1} (X_t-x_0) = \big\{L\Psi(t,X_t)+ \Phi(t,X_t)\big\}\d t
+ \partial_t^{*k_2} \int_0^t  B(s)\d W(s),\end{equation} where $W(s)$, $s\geq0$,  is cylindrical Brownian motion on $H$ and $B: [0,\infty) \to L_{HS}(H)$ is measurable and locally bounded.
 \end{Remark}
\subsection{Stochastic generalized $p$-Laplace equations}
Let $D\subset \R^d$ be an open bounded domain,  $\m$ be the normalized volume measure on $D$,  and  $p\in[2,\infty)$. Let $H_0^{1,p}(D)$ be the closure of $C_0^\infty(D)$ with respect to the norm
$$\|f\|_{1,p}:= \|f\|_p  +\|\nn f\|_p,$$ where $\|\cdot\|_p$ is the norm in $L^p(\m)$.  Let $H=L^2(\m)$ and $V=H_0^{1,p}(D)$. By the Poincar\'e inequality, there exists a constant $C>0$ such that
$\|f\|_{1,p}\le C \|\nabla f\|_{p}.$ Now we consider the following  time-fractional  generalized $p$-Laplace equations
\begin{equation}\label{PLE}
 \partial_t^{*k} (X_t-x_0) =  {\rm div}\left(\Phi(t, \nabla X_t)\right) + f(t,X_t) ,
\end{equation}
where
$$\Phi: [0,\infty)\times \R^d \to \R^d;  ~  f: [0,\infty)\times \R \to \R  $$
are measurable,  and continuous in the second variable.

To verify conditions $(H1)$, $(H2)$, $(H3)$ and $(H4)$ for $A(t,v):= {\rm div}\big(\Phi(t, \nabla v)\big) +f(t,v)$, we assume
that for a fixed $p\in[2,\infty)$,
\begin{equation}\label{7.2}\begin{split} & |\Phi(t,s)|\le K(1+|s|^{p-1}), \ \ s\in\R^d, t\ge 0,\\
&  \m\big((\Phi(t,x)-\Phi(t,y)) (x-y)\big)  \ge
\delta \|x-y\|_{p}^{p},\ \  t\ge 0, \\
&   \m\big((f(t,x)-f(t,y)) (x-y)\big)  \le
K \|x-y\|_{2}^{2},  \  t\ge 0, \\
& |f(t,x)| \le K (1+ |x|^{p-1}),  \ t\ge 0,
\end{split} \end{equation} hold for some constants $K, \delta>0$ and all $x,y\in L^{p}(\m)$.

~~

\noindent\textbf{Example  7.2} Suppose that \eqref{7.2} holds, then $(H1)$, $(H2)$, $(H3)$ and $(H4)$ hold for (see e.g. \cite[Example 4.1]{G13})
$$ A(t,v):= {\rm div}\left(\Phi(t, \nabla v)\right) +f(t,v). $$
 Therefore,  Theorem \ref{thm:2.2} is applicable to the  time-fractional  generalized $p$-Laplace equations (\ref{PLE}), if $k$ satisfies \ref{cond:k}, \ref{cond:k2} respectively.
\begin{Remark} (i)
Obviously, the assumptions above are satisfied provided
$\Phi(t,s)= h(t) |s|^{p-2}s $ and $f(t,s)= f_1(t) s -f_2(t)|s|^{p-2}s$, $t\in[0,T]$, $s\in\R$, with
$0<\inf h\le \sup h<\infty$ and $ \|f_i\|_\infty<\infty$, which is the classical $p$-Laplace equation with polynomial type perturbation.

(ii) Similarly, we could apply Theorem \ref{thm:2.3} to the following time-fractional stochastic generalized $p$-Laplace equations
\begin{equation}\label{SPLE}
 \partial_t^{*k_1} (X_t-x_0) = \left( {\rm div}\left(\Phi(t, \nabla X_t)\right) + f(t,X_t) \right) \d t + \partial_t^{*k_2} \int_0^t B(s)\d W(s),
\end{equation}
where $W(s)$, $s\geq0$, is cylindrical Brownian motion on $H$,  $B: [0,\infty)  \to L_{HS}(H)$ is measurable and locally bounded.
\end{Remark}
\subsection{Stochastic generalized fast-diffusion equations}
Let $(E,\mathcal{B},\m), (L,\D(L)), H$ ($B$ and $W(s)$, $s\geq0$,) be as in Example 7.1.
Suppose that $r\in (0,1)$ and $\Psi: [0,\infty)\times \R\to\R$ is measurable, continuous in the second variable and such that for some constant $\delta>0$,
\begin{eqnarray}
 && \big(\Psi(t,s_1)-\Psi(t,s_2)\big)(s_1-s_2)\ge \ff{ \delta |s_1-s_2|^2}{(|s_1|\lor |s_2|)^{1-r}},  \ \ \ s_1,s_2\in \R, t\ge
0,\label{CC1}\\
&& s\Psi(t,s)\ge \delta |s|^{r+1},\ \ \sup_{t\in [0,T],s\ge 0}\ff{|\Psi(t,s)|}{1+|s|^r}<\infty,\ \  s \in \R, t\ge
0,\label{CC2} \end{eqnarray} where  $\ff{|s_1-s_2|^2}{(|s_1|\lor |s_2|)^{1-r}}:=0$ for $s_1=s_2=0.$

We consider the following  time-fractional  generalized fast-diffusion equations
\begin{equation}\label{FD}
\partial_t^{*k}( X(t)-x_0) = L\Psi(t,X(t))+ h(t) X(t) ,\end{equation}
where  $ h \in C([0,\infty))$.

Let $V=L^{r+1}(\m)\cap H$ with $\|v\|_{V}:= \|v\|_{1+r} + \|v\|_{H}$. Then it is easy to show that
 $(H1)$-$(H4)$ hold for (see \cite[Theorem 3.9]{RRW} for a more general result)
 $$A(t,v):= L \Psi(t,v)+ h(t) v,\ \ \ v\in  V.$$

\noindent\textbf{Example 7.3} Suppose that $(\ref{CC1})$ and $(\ref{CC2})$ hold, then the assertions in Theorem  \ref{thm:2.2} hold for (\ref{FD}), if $k$ satisfies \ref{cond:k}, \ref{cond:k2} respectively.
\begin{Remark}
 (i) By the mean-valued
theorem,  one has  for $r\in (0,1)$
$$(s_1-s_2)(s_1|s_1|^{r-1}-s_2|s_2|^{r-1})\ge
r |s_1-s_2|^2(|s_1|\lor |s_2|)^{r-1},\ \ s_1,s_2\in\R.$$ So, a  simple example of $\Psi$, so that \eqref{CC1} and \eqref{CC2} hold, is  $\Psi(t,s)= c\,s|s|^{r-1}$  for some constant $c>0$. This corresponds to the classical fast-diffusion equation.

(ii) Similar results also hold for the corresponding  time-fractional  stochastic equations
\begin{equation}
\partial_t^{*k_1}( X(t)-x_0) = \Big\{L\Psi(t,X(t))+ h(t) X(t)\Big\}\d t +\partial_t^{*k_2} \int_0^t B(s)\d W(s).\end{equation}
\end{Remark}
\appendix
\noindent
\section{Appendix A.~ Proof of Theorem \ref{thm:4.2}}
For the proof of Theorem \ref{thm:4.2} we need some preparations.
We recall the definition of a pseudo-monotone operator, which is a very useful generalization of monotone operator and was first introduced by Br\'{e}zis in \cite{Br68}.
We use the  notation ``$\rightharpoonup$'' for  weak convergence in Banach spaces.

\begin{Definition}\label{definition:a.1} An operator $M: \mathcal{V}\longrightarrow \mathcal{V}^*$ is called pseudo-monotone  if $v_n\rightharpoonup v$ in $\mathcal{V}$
as $n\rightarrow \infty$ and
$$     \limsup_{n\rightarrow\infty} { }_{\mathcal{V}^*}\<M(v_n), v_n-v\>_{\mathcal{V}}\le 0 $$
implies for all $u\in \mathcal{V}$
$$   { }_{\mathcal{V}^*}\<M(v), v-u\>_{\mathcal{V}} \le  \liminf_{n\rightarrow\infty} { }_{\mathcal{V}^*}\<M(v_n), v_n-u\>_{\mathcal{V}}.  $$
 \end{Definition}

\begin{Remark}\label{remark:a.2}
\begin{enumerate}[label=(\roman*), leftmargin=0.7cm]
\item Browder introduced a slightly different definition of  a pseudo-monotone operator in \cite{Bro77}:
An operator $M: \mathcal{V}\longrightarrow \mathcal{V}^*$ is called pseudo-monotone   if
$v_n\rightharpoonup v$ in $\mathcal{V}$ as $n\rightarrow \infty$ and
$$     \limsup_{n\rightarrow\infty} { }_{\mathcal{V}^*}\<M(v_n), v_n-v\>_{\mathcal{V}}\le 0 $$
implies
$$   M(v_n)\rightharpoonup M(v)\ \    \text{and}   \  \
\lim_{n\rightarrow\infty} { }_{\mathcal{V}^*}\<M(v_n), v_n\>_{\mathcal{V}}={ }_{\mathcal{V}^*}\<M(v), v\>_{\mathcal{V}}.  $$
In particular, if $M$ is bounded on bounded sets, then  these two definitions are equivalent, we refer to \cite{LR13,LR15}.
\item We recall that as mentioned before our operator $\mathcal A \colon \mathcal{V} \longrightarrow \mathcal{V}^*$ in \eqref{eq:2.2'} is coercive and bounded as well as monotone if $C_1=0$ in \ref{cond:H2}, hence in particular pseudo-monotone.
If we add a continuous monotone linear operator $\tilde \Lambda \colon \mathcal{V} \longrightarrow \mathcal{H}$ to it, it is easy to see that also $\mathcal A + \tilde \Lambda$ is pseudo-monotone.
\zx \item Despite the fact that, of course, $(-\Lambda,\; D(\Lambda,\mathcal{H}))$ is maximal monotone as an operator on $\mathcal{H}$, the map $-\Lambda\colon \mathcal{F}\subset \mathcal{V}\rightarrow \mathcal{V}^*$ (or $-\Lambda^k\colon \mathcal{F}\subset \mathcal{V}\rightarrow \mathcal{V}^*$), may be not maximal monotone. Hence we cannot apply \cite{TODO} to conclude that $\mathcal{A}-\Lambda\colon \mathcal{F}\subset \mathcal{V}\rightarrow \mathcal{V}^*$ is maximal monotone. Otherwise, because by Theorem \ref{thm:2.1}(ii) $\mathcal{A}-\Lambda$ is coercive, the assertion of Theorem \ref{thm:4.2} would follow easily.
\xz
\end{enumerate}
\end{Remark}

\begin{Lemma}\label{lemma:a.2}
 If $ M: \mathcal{V} \longrightarrow \mathcal{V}^*$ is pseudo-monotone, bounded on bounded sets and coercive, then $M$ is surjective,
  i.e. for any $f\in \mathcal{V}^*$, the equation $M u=f$ has a solution.
\end{Lemma}
\begin{proof}
This is a classical result due to Br\'{e}zis. For the proof we refer to \cite{Br68} or \cite[Theorem 27.A]{Z90}.
\end{proof}
\begin{proof}[Proof of Theorem \ref{thm:4.2}]
\textbf{Step 1}:  Let  $\alpha>0$  and  consider the Yosida approximation $\Lambda_\alpha \colon \mathcal{V}\longrightarrow \mathcal{V}^*$ defined by
$$ {}_{\mathcal{V}^*}\<\Lambda_\alpha u, \cdot \>_{\mathcal{V}}:= \alpha{}_{\mathcal{V}^*}\<\alpha V_\alpha u -u, \cdot\>_{\mathcal{V}},  $$
where $V_\alpha=(\alpha-\Lambda)^{-1}$, $\alpha > 0$, is the resolvent of $( \Lambda, D(\Lambda, \mathcal H))$ (on $\mathcal H$).

We note that since $\alpha V_\alpha$ is a contraction on $\mathcal{H}$, we have
\begin{align*}
 {}_{V^*} \langle \Lambda_\alpha u, u \rangle_{V} =  \langle \Lambda_\alpha u, u \rangle_{H} \geq 0 \quad \text{for all } u \in V,
\end{align*}
hence by Remark \ref{remark:a.2}(ii) it follows that $\mathcal{A}-\Lambda_\alpha$ is pseudo-monotone, coercive and bounded on bounded sets.
Therefore, by Lemma \ref{lemma:a.2} there exists $u_\alpha\in \mathcal{V}$ such that $\mathcal{A} u_\alpha- \Lambda_\alpha u_\alpha=f$.

\smallskip
\noindent
\textbf{Step 2}: Note that
$$  { }_{\mathcal{V}^*}\<\mathcal{A} u_{\alpha}, u_{\alpha}\>_{\mathcal{V}}\le  { }_{\mathcal{V}^*}\<\mathcal{A} u_{\alpha}-\Lambda_\alpha u_\alpha, u_{\alpha}\>_{\mathcal{V}}= { }_{\mathcal{V}^*}\<f, u_{\alpha}\>_{\mathcal{V}} \le \|f\|_{\mathcal{V}^*} \|u_\alpha\|_{\mathcal{V}}.  $$
Hence, by the coercivity assumption \ref{cond:H3} we obtain that $\sup_{\alpha>0} \|u_\alpha\|_{\mathcal{V}} < \infty$, and hence
\begin{align*}\sup_{\alpha>0} \|\mathcal{A} u_\alpha\|_{\mathcal{V}^*} < \infty\end{align*} by \ref{cond:H4}.

Since for any $v\in \mathcal{V}$
 \begin{equation*}
 \begin{split}
    { }_{\mathcal{V}^*}\<\Lambda_\alpha u_{\alpha}, v\>_{\mathcal{V}}&=-{ }_{\mathcal{V}^*}\<\mathcal{A} u_\alpha-\Lambda_\alpha u_{\alpha}, v\>_{\mathcal{V}} + { }_{\mathcal{V}^*}\<\mathcal{A} u_\alpha, v\>_{\mathcal{V}} \\
    &=-  { }_{\mathcal{V}^*}\<f, v\>_{\mathcal{V}} + { }_{\mathcal{V}^*}\<\mathcal{A} u_\alpha, v\>_{\mathcal{V}} \\
                                         &\le (\|f\|_{\mathcal{V}^*}+\|\mathcal{A} u_\alpha\|_{\mathcal{V}^*})\|v\|_{\mathcal{V}},
 \end{split}
\end{equation*}
we have  $\sup_{\alpha>0} \|\Lambda_\alpha u_\alpha\|_{\mathcal{V}^*} < \infty$.

By the apriori estimates above we know there exists a subsequence $\alpha_n \rightarrow \infty$ such that
\begin{equation*}
 \begin{split}
    u_{\alpha_n} &\rightharpoonup u\  \ \text{in}\  \mathcal{V} ; \\
    \mathcal{A} u_{\alpha_n} &\rightharpoonup h\  \ \text{in}\  \mathcal{V}^* ; \\
    \Lambda_{\alpha_n}  u_{\alpha_n} &\rightharpoonup  g\  \ \text{in}\  \mathcal{V}^*.
 \end{split}
\end{equation*}
So, it is easy to see that $h-g=f$.

By the strong continuity of the dual resolvent $(\hat{V}_\alpha)_{\alpha > \omega}$ in $\mathcal{V}^*$, we have for all $v \in \mathcal{V}^*$
 $$ \lim_{n \rightarrow \infty} {}_{\mathcal{V}^*} \< v, \alpha_n V_{\alpha_n} u_{\alpha_n} \>_{\mathcal{V}} = \lim_{n \rightarrow \infty} {}_{\mathcal{V}^*} \< \alpha_n \hat{V}_{\alpha_n}v, u_{\alpha_n} \>_{\mathcal{V}} = {}_{\mathcal{V}^*}\<v,u\>_{\mathcal{V}},$$
and, therefore,
 $$ \alpha_n V_{\alpha_n}   u_{\alpha_n} \rightharpoonup u\  \ \text{in}\  \mathcal{V} .$$
Since $ \Lambda  \alpha_n V_{\alpha_n}   u_{\alpha_n}=\Lambda_{\alpha_n}  u_{\alpha_n}  $, we also have
 $$ \Lambda  \alpha_n V_{\alpha_n}   u_{\alpha_n} \rightharpoonup  g\  \ \text{in}\  \mathcal{V}^* .$$
 Since $\Lambda$ is linear and $(\Lambda, \mathcal{F})$ is closed as an operator from $\mathcal{V}$ to $\mathcal{V}^*$, this implies that $u \in \mathcal{F}$ and $\Lambda u=g$.

\smallskip
\noindent
\textbf{Step 3}: Now we only need to show $\mathcal{A} u=h$.
Since  $u_{\alpha_n} \rightharpoonup u$ in $\mathcal{V}$ and for all $v \in D(\Lambda, \mathcal{V})$
\begin{align*}
 \limsup_{n \rightarrow \infty} {}_{\mathcal{V}^*}\<\Lambda_{\alpha_n} u_{\alpha_n}, u_{\alpha_n}\>_\mathcal{V}
&=  \limsup_{n \rightarrow \infty} ({}_{\mathcal{V}^*}\<\Lambda_{\alpha_n} u_{\alpha_n}, u_{\alpha_n}-v\>_\mathcal{V} + {}_{\mathcal{V}^*}\<\Lambda_{\alpha_n} u_{\alpha_n},v\>_\mathcal{V}) \\
&\le \limsup_{n \rightarrow \infty} {}_{\mathcal{V}^*}\<\Lambda_{\alpha_n} v, u_{\alpha_n}-v\>_\mathcal{V} + {}_{\mathcal{V}^*}\<\Lambda u, v\>_\mathcal{V} \\
&= {}_{\mathcal{V}^*}\<\Lambda v, u-v\> _\mathcal{V} + {}_{\mathcal{V}^*}\<\Lambda u, v \>_\mathcal{V} ,
\end{align*}
 where the inequality follows from ${ }_{\mathcal{V}^*}\<\Lambda_{\alpha_n} ( u_{\alpha_n}-u ), u_{\alpha_n}-u\>_{\mathcal{V}} \le 0 $, since each $\Lambda_{\alpha_n}$ is a contraction on $\mathcal{H}$. Since $D(\Lambda, \mathcal{V})$ is dense in $(\mathcal{F}, \|\cdot\|_\mathcal{F})$, the above inequality extends to all $v \in \mathcal{F}$. In particular, we may take $v=u$, to obtain that
  $$\limsup_{n \rightarrow \infty} {}_{\mathcal{V}^*}\<\Lambda_{\alpha_n} u_{\alpha_n}, u_{\alpha_n}\>_\mathcal{V}\le {}_{\mathcal{V}^*}\<\Lambda u, u\>_\mathcal{V}.$$ Therefore,
 \begin{equation*}
 \begin{split}
   \limsup_{n\rightarrow\infty} { }_{\mathcal{V}^*}\<\mathcal{A} u_{\alpha_n}, u_{\alpha_n}-u\>_{\mathcal{V}}
  &= \limsup_{n\rightarrow\infty} { }_{\mathcal{V}^*}\<\Lambda_{\alpha_n} u_{\alpha_n}+f, u_{\alpha_n}-u\>_{\mathcal{V}} \\
  &= \limsup_{n\rightarrow\infty} { }_{\mathcal{V}^*}\<\Lambda_{\alpha_n} u_{\alpha_n}, u_{\alpha_n}-u\>_{\mathcal{V}} \\
  &\le { }_{\mathcal{V}^*}\<\Lambda u, u-u\>_{\mathcal{V}}=0 .
 \end{split}
\end{equation*}

So, we have
$$\limsup_{n\rightarrow\infty} { }_{\mathcal{V}^*}\<\mathcal{A} u_{\alpha_n}, u_{\alpha_n}\>_{\mathcal{V}} \le  { }_{\mathcal{V}^*}\< h, u \>_{\mathcal{V}}. $$
Hence, by the pseudo-monotonicity, we have for any $w\in \mathcal{V}$
 \begin{equation*}
 \begin{split}
   { }_{\mathcal{V}^*}\<\mathcal{A} u, u- w\>_{\mathcal{V}}
  &\le  \liminf_{n\rightarrow\infty} { }_{\mathcal{V}^*}\< \mathcal{A} u_{\alpha_n}, u_{\alpha_n}-w\>_{\mathcal{V}} \\
  &\le  \liminf_{n\rightarrow\infty} { }_{\mathcal{V}^*}\<\mathcal{A} u_{\alpha_n}, u_{\alpha_n}\>_{\mathcal{V}}- { }_{\mathcal{V}^*}\< h, w\>_{\mathcal{V}} \\
  &\le  { }_{\mathcal{V}^*}\< h, u-w\>_{\mathcal{V}},
 \end{split}
\end{equation*}
which implies $\mathcal{A} u=h$ since $w \in \mathcal{V}$ was arbitrary.
\end{proof}

\section*{Acknowledgements}
Financial support by the DFG through the CRC 1283 ``Taming uncertainty and profiting from randomness and low regularity in analysis, stochastics and their applications'' is acknowledged.
 W.L. is  supported by NSFC (No.~11822103,11831014,12090011) and the PAPD of Jiangsu Higher Education Institutions,
J.L.S. is supported by project
I$\&$D: UID/MAT/04674/2019.

The second named author would like to thank his hosts at Madeira University for a very pleasant stay in May 2018 and Summer 2019, where a part of this work was done.
 He would also like to thank the Isaac Newton Institute for a very stimulating stay in November 2018, where substantial progress was made on this paper.

\end{document}